\documentclass[12pt,reqno,twoside]{amsart}

\usepackage{amsmath}

\usepackage{amssymb}

\usepackage{epsfig}

\usepackage{color}

\numberwithin{equation}{section}

\newcommand{\Radu}{{\mathrm{Rad}_u}}

\font\sb = cmbx8 scaled \magstep0

\font\sn = cmssi8 scaled \magstep0

\long\def\comdima#1{\ifdraft{\bf #1 }\else\ignorespaces\fi}

\long\def\combarak#1{\ifdraft{\sb #1 }\else\ignorespaces\fi}

\long\def\commargin#1{\ifdraft{\marginpar{\it #1}}\else\ignorespaces\fi}

\newcommand\qu{quasiunipotent}
\newcommand\A{a}
\newcommand\B{b}

\newcommand\pu{\mathbf{p}_U}

\newcommand{\vf}{{\bf{f}}}
\newcommand{\vl}{{\bf{l}}}
\newcommand{\vh}{{\bf{h}}}

\newcommand{\dt}{{\mathcal D}_t}
\newcommand\mr{M_{m,n}}

\newcommand\pt{\Phi_t}

\newcommand\amr{$Y\in M_{m,n}
$}

\newcommand\ba{badly approximable}

\newcommand\da{Diophantine approximation}
\newcommand\di{Diophantine}

\newcommand\hs{homogeneous space}

\newcommand\ehs{expanding horospherical subgroup}

\newcommand{\Ker}{\operatorname{Ker}}

\newcommand\ua{\hat a}
\newcommand\oa{\check a}
\newcommand\ut{\hat t}

\newif\ifdraft\drafttrue

\draftfalse

\newcommand\name[1]{\label{#1}{\ifdraft{\sn [#1]}\else\ignorespaces\fi}}
\newcommand\bname[1]{{\ifdraft{\sn [#1]}\else\ignorespaces\fi}}

\newcommand{\Id}{{\operatorname{Id}}}

\newcommand\eq[2]{{\ifdraft{\ \tt [#1]}\else\ignorespaces\fi}\begin{equation}\label{eq:#1}{#2}\end{equation}}

\newcommand {\equ}[1]     {\eqref{eq:#1}}
\newcommand{\under}[2]{\underset{\text{#1}}{#2}}

\newcommand{\goth}[1]{{\mathfrak{#1}}}

\newcommand\g{\goth g}
\newcommand\h{\goth h}

\newcommand{\Q}{{\mathbb {Q}}}

\newcommand{\vr}{{\bf r}}

\newcommand{\vs}{{\bf s}}

\newcommand{\BA}{{\bold{Bad}}}

\newcommand{\rank}{{\mathrm{rank}}}

\newcommand{\R}{{\mathbb{R}}}

\newcommand{\Z}{{\mathbb{Z}}}

\newcommand{\C}{{\mathbb{C}}}

\newcommand{\N}{{\mathbb{N}}}

\newcommand{\cl}{\overline}

\newcommand{\vv}{{\bf{v}}}

\newcommand{\Ad}{{\operatorname{Ad}}}
\newcommand{\ad}{{\operatorname{ad}}}

\newcommand{\GL}{\operatorname{GL}}

\newcommand{\SL}{\operatorname{SL}}

\newcommand{\ggm}{G/\Gamma}

\newcommand{\Lie}{\operatorname{Lie}}

\newcommand{\diag}{{\rm diag}}

\newcommand{\End}{{\rm End}}

\newcommand {\ignore}[1]  {}

\newcommand{\spa}{{\rm span}}

\newcommand{\dist}{{\rm dist}}

\newcommand{\diam}{{\rm diam}}

\newcommand{\df}{{\, \stackrel{\mathrm{def}}{=}\, }}

\newcommand{\FF}{{\mathcal{F}}}

\newcommand{\x}{{\mathbf{x}}}

\newcommand{\vu}{{\bf u}}

\newcommand{\vw}{{\bf w}}

\newcommand{\p}{{\bf p}}

\newcommand{\vp}{{\bf p}}
\newcommand{\vq}{{\bf q}}

\newcommand{\til}{\widetilde}

\newcommand{\sm}{\smallsetminus}

\newcommand{\vre}{\varepsilon}

\newcommand\hd{Hausdorff dimension}

\newcommand\nz{\smallsetminus \{0\}}
\newcommand{\ay}{{\bf{A}_\infty}}

\newcommand\td{tessellation domain}
\newcommand\tn{tessellation}

\newcommand{\ca}{{\mathcal A}}
\newcommand{\cd}{{\mathcal D}}

\newtheorem{thm}{Theorem}[section]

\newtheorem{lem}[thm]{Lemma}

\newtheorem{prop}[thm]{Proposition}

\newtheorem{cor}[thm]{Corollary}

\title[Modified Schmidt games]{Modified Schmidt games 
 and a conjecture of Margulis}

\author{Dmitry Kleinbock}

\address{Brandeis University, Waltham MA
02454-9110 {\tt kleinboc@brandeis.edu}}

\author{Barak Weiss}

\address{Ben Gurion University, Be'er Sheva, Israel 84105
{\tt barakw@math.bgu.ac.il}}

\date{
January 2010}

\begin{document}

 \begin{abstract} We prove a conjecture of G.A.~Margulis on the abundance of certain 
exceptional orbits of partially hyperbolic flows on \hs s by utilizing a theory
of modified Schmidt games, which
 are modifications of $(\alpha,\beta)$-games introduced by \linebreak 
W.~Schmidt in mid-1960s.
\end{abstract}

\maketitle
\section{Introduction}
Let $X$ be a separable metric space and $F$ a group or semigroup
acting on $X$ continuously. Given $Z \subset X$, denote
by $E(F,Z)$ the set of points of $X$ with $F$-orbits staying away from
$Z$, that is, 
$$
E(F,Z) \df \{x \in X : \overline{Fx}\cap
Z = \varnothing \}\,.
$$
If $Z$ consists of a single point $z$, we will write $
E(F,z)$ instead of 
$E(F,\{z\})$.
Such sets are clearly null with respect to any $F$-ergodic measure of full support.
Similarly, denote by $E(F,\infty)$ the set of 
points of $X$ with bounded $F$-orbits;  it is also a null
set with respect to a measure as above if 
$X$ is not compact. However in many `chaotic' situations those sets can be quite big, see
\cite{Do} and \cite{K} for history and references. 

\smallskip

In this paper we take $X = \ggm$, where $G$ is a 
Lie group and $\Gamma$ a lattice in $G$, and consider
\eq{f}{F = \{g_t: t\in\R\}\subset G} acting on $X$ by left translations. 
One of the major developments in the theory of homogeneous flows is a series of
celebrated results of Ratner 
(see the
surveys \cite{Ratner - ICM94, handbook} and a book \cite{Morris})
verifying  conjectures of Raghunatan, 
Dani and Margulis on orbit closures and 
invariant measures of unipotent flows.
Ratner's theorems imply that  for quasiunipotent
subgroups ($F$ is {\sl quasiunipotent\/} if all the eigenvalues of $\Ad(g_1)$ are of 
modulus $1$, and  {\sl non-quasiunipotent\/}  otherwise), the sets $E(F,Z)$ are 
countable unions of submanifolds admitting an explicit algebraic
description, see \cite{Starkov}.

In his 1990 ICM plenary
 lecture, just before Ratner's announcement of
her results, Margulis formulated a list of conjectures
on rigidity properties of unipotent flows. In the same talk  
Margulis conjectured that such rigid behavior is absent if $F$ is 
non-quasiunipotent, providing a list of `non-rigidity conjectures for non-unipotent
flows'. For example, he conjectured \cite[Conjecture (A)]{Ma} that $E(F,\infty)$ 
is a {\sl thick\/} set; that is, its intersection with any open subset of
$X$ has full \hd. In fact, to avoid obvious counterexamples, 
one should assume that the flow $(\ggm,F)$  is  {\sl absolutely
non-\qu\/}, that is, it has no factors $(G'/\Gamma',F')$ 
(i.e.,  homomorphisms $p:G\to G'$ with $F' = p(F)$ and $\Gamma' = p(\Gamma)$ a lattice in $G'$)
such that $F'$ is  nontrivial and \qu. In 1996 Margulis and the
second-named author \cite{KM} proved Conjecture (A) of Margulis under
this hypothesis.   

\ignore{  say that 
 $(G/\Gamma,F)$ {\sl has property\/} (Q) if for any connected
normal subgroup $N\subset G$ with the quotient map $p:G\to G'\df G/N$,
for which $G'$ is semisimple without
compact factors and $p(\Gamma)$ is an irreducible lattice in $G'$,
 at least one of the following
three conditions is satisfied:

\begin{itemize}
\item[(Q1)] 
 $p(\Gamma)$ is cocompact in $G'$;

\item[(Q2)]  Ad$\,p(F)$ is relatively compact;

\item[(Q3)] $p(F)$ is not \qu.

\end{itemize}

It is proved in \cite{KM} that $E(F,\infty)$  is  {\it thick\/} in $\ggm$ if and only if 
$(G/\Gamma,F)$ has property (Q).

\begin{itemize}
\item[($*$)] $G$ is a connected semisimple Lie group without compact
factors, $\Gamma$ is an irreducible lattice in $G$, and $F$  is not {\it \qu\/}, that is,
Ad$\,g_1$ has an eigenvalue with modulus different from $1$. 
\end{itemize}
Then  
$E(F,\infty)$  is  {\it thick\/}, that is, it 
  has full \hd\ at any point of $\ggm$. This was conjectured by G.A.\,Margulis
 in his 1990 ICM address \cite[Conjecture (A)]{Ma} and proved by Margulis and the second-named author \cite{KM}
in 1996. \commargin{We should say later what exactly was conjectured and what was
 proved.}}

Note that even earlier, in the mid-1980s, Dani established the thickness 
of  $E(F,\infty)$ 
in the following two 
special cases \cite{Dani-div, Dani-rk1}:
\eq{case1}{\begin{aligned}
G = \SL_k({\R}),\ \Gamma = \SL_k({\Z}),\text{  and }
F = \{g_t\}\,,\\
\text{ where }g_t =
\text{diag}(e^{ct},\dots , e^{ct}, e^{-d t},\dots , e^{-d
t})\end{aligned}}
(here $c,d > 0$ are chosen so that the determinant of $g_t$ is $1$), and
\ignore{\begin{itemize}
\item[($*$)] \cite{Dani-div} $G = \SL_k({\R})$, $\Gamma = \SL_k({\Z})$, and 
$F = \{g_t\}$, where
$$
g_t =
\text{diag}(e^{ct},\dots , e^{ct}, e^{-d t},\dots , e^{-d
t})\,,
$$
and $c,d > 0$ are such that the determinant of $g_t$ is 1;
\item[($**$)]  }
\eq{case2}{\begin{aligned}
\text{$G$ is a connected semisimple Lie group
with }\rank_{\R} (G) =1.\end{aligned}}

In both cases Dani's approach consisted in 
studying the set of points $x\in\ggm$ with bounded one-sided trajectories $\{g_tx : t\ge 0\}$,
and showing that its intersection with any orbit of 
a certain subgroup of $G$ is a
{\sl winning set\/} for a  
game invented by Schmidt \cite{Schmidt games}.  
This property is stronger than thickness
in the sense that a countable intersection of winning sets is also winning; 
see \cite{Dani-games} for a nice survey of Dani's results.

Understanding one-sided trajectories also plays a crucial role in  
the approach of  \cite{KM}, as well as in 
the present paper.
In what follows we will reserve the notation $F$ for one-parameter subgroups as in \equ{f},
and let $F^+ = \{g_t : t \ge 0\}$. 
\ignore{Before stating it we introduce the following 
representation-theoretic property.
Let $G$ be a connected Lie group, $H$ a closed subgroup, and $F = \{g_t : t \ge 0\}$ a one-parameter
subsemigroup of $G$. 
Given a finite-dimensional representation $\tau: G \to \GL(V)$,
define
$$V^< = \left\{v \in V: \tau(g_t)v \to_{t \to \infty} 0 \right\},$$
and say that a subgroup  $H$ of $G$ is {\em $(F, \tau)$-expanding} if 
\eq{eq: expanding condition}{
v \in V, \ \tau(H)v \subset V^< \ \ \implies \ \ v=0, 
}
i.e. for any nonzero $v$ which shrinks under $\tau(g_t)$ we can find
$h \in H$ such that $\tau(h)v$ does not shrink. 
We will say that $H$ is {\em $F$-expanding} if it is $(F,
\rho)$-expanding for every $\rho$ which is an exterior power of the 
adjoint  representation of $G$. \commargin{ before it was: where $\rho$ is the representation of $G$ on the
subspaces of its Lie algebra via the exterior powers...  } An  example of an $F$-expanding subgroup is 
the }
Also throughout the paper we will denote by $H^+$ 
 the
{\sl expanding horospherical subgroup\/} corresponding to $F^+$,
defined by 
\eq{defn of U+}{
H^+ \df \{g \in G: g_{-t} g g_{t} \to_{t \to +\infty} e\}\,.
}
Note that $H^+$ is nontrivial if and only if $F$ is non-quasiunipotent. 
It is a connected simply connected
nilpotent Lie group normalized by $g_t$. Moreover, it admits 
a one-parameter semigroup of {\sl contracting automorphisms\/}
\eq{def phi}{\FF = \{\Phi_t : t > 0\}, \ \ \ \mathrm{where} \ \ 
\Phi_t(g) = g_{-t} g g_{t}
\,.}   

The main step of the proof in \cite{KM} was to show that, in the special case 
when $G$ is a connected semisimple Lie group without compact
factors and  $\Gamma$ is an irreducible lattice in $G$,  
for any $x\in \ggm$ the set 
\eq{efg}{\{h\in H^+ : hx \in  E(F^+,\infty)\}}
is thick. 
In  the present paper we strengthen the above conclusion,
replacing thickness with the winning property for a certain game. 
More precisely, in 
\cite{games} the notion of {\sl
modified Schmidt games\/}   
(to be abbreviated by MSGs)
induced by semigroups of contracting automorphisms
was introduced. These generalize the game devised by Schmidt and share many of its 
appealing properties. In particular, winning sets of these games are thick and  
the countable intersection property also holds. All the relevant definitions
and facts are discussed in  \S \ref{sec:  games}.
%
The following is one of 
our main results:

\begin{thm}\name{thm: C} Let  
$G$ be a connected semisimple centerfree Lie group without compact factors,
$\Gamma$ an irreducible lattice in $G$, $F$ a one-parameter non-\qu \
subgroup of $G$. Then there exists $a' > 0$ such that for any $x\in \ggm$ and $\A > a'$,
the set  \equ{efg} \commargin{Dima: funny, but we don't seem to be able to prove that  $a'$-winning 
implies  $\A$-winning
if  $\A > a'$.}
is an $\A$-winning set for the MSG  
induced by $\FF $ as in \equ{def phi}. 
\end{thm}

We also show

\begin{thm}\name{thm: B}
Let $G$, $F$ and  $\Gamma$ be as in Theorem \ref{thm: C}.
Then there exists $a'' > 0$ such that for any $x,z\in\ggm$ and $\A > a''$, the set
$$\big\{h\in H^+ : hx \in  E(F^+,z)\big\}$$
is an $\A$-winning set for the  MSG 
induced by $\FF$ as in \equ{def phi}. 
\end{thm}

This strengthens the results of \cite{K}, where the above sets 
were shown to be thick. 
%
%
Because of the 
countable intersection property, Theorems \ref{thm: C} and \ref{thm: B} readily
imply the following \commargin{Dima: removed  $X_0$ from the statement to simplify things.}

\begin{cor}\name{cor: D} Let $G$, $\Gamma$ and $F$ 
be as in Theorem \ref{thm: C}, and let $
Z$ be a 
countable subset of $\ggm$. Then for any $x\in\ggm$, 
the set
\eq{efzinftyh}{
\big\{h\in H^+ : hx \in   E(F^+,Z)  \cap E(F^+,\infty) \big\}
}
is 
a winning set for the MSG 
induced by $\FF$. 
\end{cor}

This extends a result of Dolgopyat
\cite[Corollary 2]{Do} who
established the thickness of sets 
$$\big\{h\in H^+ : hx \in   E(F^+,Z)  \cap E(F^+,\infty) \big\}$$
for countable $Z$ in the case when the $F^+$-action corresponds to a geodesic
flow on the unit tangent bundle to a manifold of constant negative
curvature and finite volume. 

\smallskip
From Corollary \ref{cor: D} we deduce

\begin{thm}\name{cor: F} Let $G$ be a 
Lie group, $\Gamma$ a
lattice in $G$, $F$ a one-parameter absolutely non-\qu \ subgroup of
$G$, and $
Z$ a countable subset of $\ggm$. 
Then the set
\eq{efzinfty}{
\left\{x \in E(F,Z)  \cap E(F,\infty): \ 
\dim(\overline{Fx}) < \dim(G)\right\}
}
is thick.
\end{thm}

Here and hereafter $\dim(X)$ stands for the \hd\ of a metric space $X$.
The above theorem settles Conjecture (B)
made by  Margulis in the aforementioned ICM address \cite{Ma}; in fact, that conjecture was about 
  {\it  finite\/} $Z\subset \ggm$.  
 To reduce Theorem \ref{cor: F} to Corollary \ref{cor: D}, we employ
a standard (see e.g.\ \cite{Dani-div}) reduction from $G$ semisimple to
general $G$,  pass from one-sided to two-sided orbits roughly following the
lines of \cite[\S 1]{KM}, and also use an entropy argument, which
we learned from  
Einsiedler and 
Lindenstrauss, to show that
the condition $\dim(\overline{Fx}) < \dim(G)$ is automatic
as long as $x\in E(F,Z) \cap E(F,\infty)$.

 \medskip

Both Theorem \ref{thm: C} and Theorem \ref{thm: B} will be proved in a
stronger form. Namely, we establish the conclusion of Theorem
\ref{thm: C} with $H^+$ replaced with its certain proper subgroups, so-called 
$(F^+,\Gamma)$-expanding subgroups, introduced and studied in
\S\ref{section: F expanding groups}.\ignore{
\begin{thm}\name{thm: more general} Suppose 
$G$ is a semisimple Lie group with no compact factors, $\Gamma$ is an
irreducible lattice and $F$ is a non \qu \ one-parameter subgroup. Let
$U$ be an $(F^+,\Gamma)$-expanding subgroup of $H^+$ normalized by $F$. 
Then for any $x\in \ggm$,
the set $\{h\in U : hx \in  E(F^+,\infty)\}$ 
is a winning set for the MSG
induced by  $\FF $ as in \equ{def phi}. 
\end{thm}}
We will show that $H^+$ is always $(F^+,\Gamma)$-expanding; thus Theorem \ref{thm: C}
follows from  its stronger version, Theorem \ref{thm: more general}. 
Such a generalization is useful for other applications as well.
Namely, in \S\ref{concl} we apply Theorem \ref{thm: more general} to \da\ with weights, 
strengthening results obtained earlier in  \cite{di, games}.
\ignore{ let $m,n$ be positive integers,
denote by $\mr$ the space of $m\times n$ matrices with real entries and
choose $\vr\in\R^m$ and $\vs\in\R^n$ such that
\eq{def rs}{r_i,s_j > 0\quad\text{and}\quad\sum_{i=1}^m r_i = 1 =
 \sum_{j=1}^n s_j\,.}
One says that  \amr\ is  
 {\sl $(\vr,\vs)$-\ba\/}, denoted by $Y\in\BA(\vr,\vs)$, if
 \eq{def bars}{ \inf_{\vp\in\Z^m,\,\vq\in\Z^n\nz} \max_i |Y_i\vq - p_i|^{1/r_i} \cdot
\max_j|q_j|^{1/s_j}    > 0\,,
}
where $Y_i$, $i = 1,\dots,m$ are rows of $Y$. A dynamical 
interpretation of this property  \cite{Dani-div, K-matrices}  is used to deduce
the following statement from Theorem \ref{thm:
more general}: 

\begin{cor}\name{cor: E} $ \BA(\vr,\vs)$
is winning for the MSG induced by the semigroup of contractions
$\Phi_t : (y_{ij}) \mapsto ( e^{- (r_i + s_j)t }y_{ij})$ of $\mr$. 
 \end{cor}

The case $n=1$ of this corollary is the main result of \cite{games}, which was proved via 
a variation of Schmidt's methods, not using homogeneous dynamics. The fact that the
set $ \BA(\vr,\vs)$ is thick was known before, see \cite{PV-bad} for
the case $n=1$ and \cite[Corollary 4.5]{di} for the general case. 
}
\smallskip

As for Theorem \ref{thm: B}, 
its conclusion in fact holds for any \hs\ $\ggm$, where $G$ is a Lie group and 
 $\Gamma\subset G$ is discrete. Further, we are able to replace $H^+$ by any nontrivial
connected subgroup $H\subset H^+$ normalized by $F$, see Theorem \ref{thm: B more general}
for a more general statement. 
\ignore{

\begin{cor}\name{cor: geodesic}
Let $M$ be a complete  Riemannian
manifold of constant negative curvature, let  $\{z_k\}$ be its countable subset, and let $F$ be the 
group $\R$ acting on the unit tangent bundle $S(M)$ via the geodesic flow.  Then for any
$y\in M$, the set of directions
$$
\left\{\xi: 
(y,\xi) \in E(F, \infty) \cap E\big(F, \cup_k S(z_k)\big)
\right\}\,,
$$ 
i.e.\ those for which the two-sided geodesic starting at $(y,\xi)$ is bounded and stays away from 
points $z_k$ for all $k$, is thick.
\end{cor}

Another} 
One application of that 
is a possibility of intersecting sets
$E(F^+,z)$ for {\it different\/} choices of groups $F$, see Corollary \ref{cor: S}.
\smallskip

The structure of the paper is as follows. 
In \S \ref{sec:  games} we describe modified Schmidt games,
the main tool designed to prove our results. Then in \S\ref{cc} we use reduction theory and 
geometry of rank-one \hs s to state criteria for compactness of subsets of $\ggm$. The next section
contains the definition, properties and examples of $(F^+,\Gamma)$-expanding subgroups.
There we also state Theorem \ref{thm: more general} 
which is proved in \S\ref{winning}.
Then in \S\ref{esc points}
we deal with Theorem \ref{thm: B} via its stronger version, Theorem \ref{thm: B more general},
and derive Corollary \ref{cor: D}. 
The reduction of 
Theorem \ref{cor: F} to  Corollary \ref{cor: D} occupies \S\ref{margconj}. Several concluding 
remarks are made in the last section of the paper.

\medskip
\noindent{\bf Acknowledgements.} 
The authors are grateful to the hospitality of the Fields Institute for Research in Mathematical Sciences
(Toronto) and Max Planck Institute for Mathematics
(Bonn), where some parts of this work were accomplished.
Thanks are also due 
 to Manfred Einsiedler and Elon Lindenstrauss for explaining the proof of Proposition
\ref{prop: smalldim irr}. The authors were supported by BSF grant 2000247, ISF grant 584/04, and NSF
Grants DMS-0239463, DMS-0801064.
\section{Games}\name{sec: games}
In what follows, all distances (diameters of sets) in various metric spaces 
will be denoted by `dist' (`diam').
 Let $H$ be a connected 
Lie group with a right-invariant Riemannian metric, 
and assume that it admits 
a one-parameter group of 
automorphisms $
\{{\Phi}_t: t \in\R\}$ such that $\Phi_t$ is contracting for positive  $t$ 
(recall that ${\Phi}:H\to H$ is {\sl contracting\/} if 
for every $g\in H$, ${\Phi}^k(g)\to e$ as $k \to \infty$). 
%
In other words, ${\Phi}_t = \exp(tY)$, where $Y \in \End \big(\Lie(H) \big)$
 and the real parts of all eigenvalues of $Y$ are negative. Note that
$Y$ is not assumed to be diagonalizable.
We denote by $\mathcal{F}$ the semigroup  $\mathcal{F} \df \{{\Phi}_t : t > 0\}$
corresponding to positive values of $t$.
In fact, in all the examples of the present paper $H$ will be a subgroup
of $H^+$ as in \equ{defn of U+} normalized by $F$ as in \equ{f},
 and $\mathcal{F}$ will be of the form \equ{def phi}.

Say that a subset $D_0$ of $H$  is {\sl admissible\/} if it
is  compact and has non-empty interior. Fix an admissible $D_0$ and denote by $\cd_t$ the set of 
all right-translates of $\Phi_t(D_0)$, \eq{def dt}{\dt \df \{\Phi_t(D_0)h : h \in H\}\,.} 
Because the maps $\Phi_t$, $t > 0$, are contracting and $D_0$ is  admissible,
there exists $a_*$ such that for any $t > a_*$, the set of $D\in \cd_t$
contained in $D_0$ is nonempty.
Also, since $D_0$ is bounded, it follows that for some $c_0 > 0$
one has 
${\dist \big({\Phi}_t(g),{\Phi}_t(h)\big) \le c_0 e^{-\sigma t}
}$
for all $g,h\in D_0$, where $\sigma>0$ is  such that  the real
parts of all the eigenvalues of $Y$ as above are  smaller than $-\sigma$. 
Therefore 
\eq{diam squeeze}{\diam(D) \le  c_0 e^{-\sigma t}\quad\forall\,D\in \dt}
(recall that the metric is chosen to be right-invariant, so all the elements of $\dt$ are isometric to
${\Phi}_{t}(D_0)$).

\smallskip

Now pick two real numbers $\A,\B > a_*$ and, following \cite{games},
define a game, played by two players, Alice and Bob. 
\commargin{Barak: I put in Alice and Bob everywhere.}
First Bob 
picks $t_1\in\R$ and $h_1\in H$, thus defining $B_1
\df \Phi_{t_1}(D_0)h_1\in \mathcal{D}_{t_1}$. Then Alice 
picks a
translate $A_1$ of  $\Phi_\A(B_1)$ which is contained in $B_1$, Bob 
picks a translate $B_2$ of $\Phi_\B(A_1)$ which is contained in
$A_1$, after that Alice 
picks a translate 
$A_2$ of $\Phi_\A(B_2)$ which is contained in $B_2$,
and so on. In other words, for $k\in \N$ we set 
 \eq{si ti}{t_k =   t_1+(k-1)(\A+\B)\text{ and  }
t_k' =   t_1+(k-1)(\A+\B) + \A 
\,.}
Then at the $k$th step of the game, Alice 
picks $A_k \in \cd_{t_k'}$ inside $B_{k}$, and then Bob
picks $B_{k+1} \in \cd_{t_{k+1}}$ inside $A_k$ etc. Note that Alice
and Bob 
can choose such sets because $\A,\B > a_*$, and that the intersection
\eq{int}{
\bigcap_{i = 1}^\infty A_i = \bigcap_{i = 0}^\infty B_i} 
is nonempty and consists of a single point in view of 
\equ{diam squeeze}. We will refer to this procedure as to the {\sl
$\{\dt\}$-$(\A, \B)$-modified Schmidt game}, 
abbreviated as  $\{\dt\}$-$(\A, \B)$-MSG.
Let us say that $S\subset H$ is {\sl $(\A, \B)$-winning\/}
for the $\{\dt\}$-MSG 
if 
Alice 
can proceed in such a way that the point \equ{int} is contained in $S$ no matter how 
Bob 
plays. 
Similarly, say
 that $S$ is an  {\sl $\A$-winning\/} set of the game if 
 $S$ is $(\A, \B)$-winning for any choice of $\B > a_*$, and that $S$ is
{\sl winning} if it is $\A$-winning for some $\A > a_*$. 

\smallskip
The game described above 
falls into the framework of $(\frak F, \frak S)$-games introduced by
Schmidt in \cite{Schmidt games}. Taking $H = \R^n$, $D_0$ to be the 
closed unit ball in $\R^n$ and $\Phi_t = e^{-t}\Id$ corresponds to the set-up
where $\cd_t$ consists of balls of radius $e^{-t}$, the standard special case
of the construction often referred to as {\sl Schmidt's $(\alpha,\beta)$-game\/}\footnote{Traditionally
the game is described in multiplicative notation, where radii of balls are multiplied 
by fixed constants $\alpha = e^{-a}$
and $\beta = e^{-b}$.}. It was shown in \cite{games} that many features of Schmidt's
original game, established in \cite{Schmidt games}, are present in this more general situation. 
\commargin{Barak: deleted a sentence here, see Tex file.}
The next theorem summarizes the two basic properties important for our purposes:

\begin{thm}
\name{thm: countable general}  Let $H$, $\mathcal{F}$ and admissible $D_0$ be as above.

\begin{itemize}
\item[(a)]  \cite[Theorem 2.4]{games} Let $\A > a_*$,
and let $S_i\subset H$, $i\in\N$, be 
a sequence of $\A$-winning sets of the $\{\dt\}$-MSG.
Then $\cap_{i = 1}^\infty S_i$
is also  $\A$-winning.

\item[(b)] \cite[Corollary 3.4]{games} Any winning set for the $\{\dt\}$-MSG 
is thick. 
\end{itemize}
\end{thm}

\ignore{
Here is another lemma from \cite{games}:

\begin{lem}  \cite[Lemma 2.5]{games} 
\name{lem: dummy} Let $H$, $\mathcal{F}$ and admissible $D_0$ be as above,
 and suppose that $S\subset H$,  $\A, \B > a_*$ 
are such that whenever Bob 
initially chooses $B_1 \in \dt$ with $t \ge {t_0}$, Alice 
can win the $\{\dt\}$-$(\A, \B)$-MSG. 
Then $S$ is an $(\A,\B)$-winning set of the game.
\end{lem}

Indeed, Alice 
can 
make arbitrary (dummy) moves waiting until $t_k$ becomes at least ${t_0}$,
and then apply the strategy she is assumed to have. 
Consequently, the collection of  $(\A,\B)$-winning sets of the game
depends only on 
the `tail' of the family $\{\dt\}$. 

\smallskip

Finally, l}

Let us record another useful observation made in \cite{games}:

\begin{prop}  \cite[Proposition 3.1]{games}
\name{prop: initial domain} Let $D_0, D_0'$ be admissible, and  define 
 $\{\cd_t\}$ and $\{\cd'_t\}$ as in \equ{def dt} using $D_0$ and $D_0'$
respectively. Let
$s>0$ be such that 
for 
some $h,h'\in H$,
\eq{squeeze}{
 {\Phi}_{s}(D_0)h\subset D_0' 
\text{ and }  {\Phi}_{s}(D_0')h'\subset D_0
}
(such an $s$ exists in light of \equ{diam squeeze} and the admissibility of $D_0, D_0'$). 
Suppose that  
 $S$ is $a$-winning for the $\{\dt\}$-MSG, then it is $(\A+
2s)$-winning for the $\{\cd'_t\}$-MSG. 
\end{prop}

Based on the above proposition, in what follows we are going to choose 
some admissible initial domain $D_0\subset H$, not worrying about specifying it 
explicitly,  and use it to define the game.
Moreover, we will suppress $D_0$ from the notation and refer to the game
as to the modified Schmidt game {\sl induced by $\mathcal{F}$\/}. Even though the game itself,
and the value of $a_*$,
will depend on the choice of an admissible $D_0$, the class of winning sets will not. 
Hopefully this terminology will not lead to confusion.

\section{Compactness criteria}
\name{cc}
For a Lie group $G$ and its discrete subgroup $\Gamma$, we denote by  
$\goth{g}$ the Lie algebra of $G$ and by $\pi$ the quotient map $G\to \ggm$, $g\mapsto g\Gamma$.
In the next three sections we will take $G$ and $\Gamma$ as in 
Theorem \ref{thm: C}, that is, $G$ is a connected semisimple
centerfree Lie group without compact factors. It is well-known (see
e.g.\ \cite[Prop. 3.1.6]{Zimmer}) 
that under these assumptions, $G$ is isomorphic to the connected
component of the identity in the real points of an algebraic group
defined over $\Q$; in the sequel we will use this
algebraic structure without further comment. We also assume $\Gamma\subset G$ to be non-uniform
and irreducible. By the
Margulis Arithmeticity Theorem \cite[Cor. 6.1.10]{Zimmer}, 
\commargin{Barak: I changed the ref to Zimmer although I am not sure it is
better. Dima: I think it is OK.} either $\rank_\R(G) = 1$,
or $\Gamma$ is {\sl arithmetic\/}, i.e. $G(\Z)$ is
commensurable with $\Gamma$. 
In this section we review results generalizing Mahler's compactness criterion,
which give necessary and sufficient conditions for a subset of $\ggm$
to be precompact. We will discuss each of the above cases separately.


\subsection{Rank one spaces}
\name{rk1}
Let $\rank_{\R} (G) =1$. Let $P$ be a minimal $\R$-parabolic subgroup of
$G$, $U = \Radu(P)$ the unipotent radical of $P$,  
$d=\dim U$,  $V = \bigwedge^d
\goth{g}$. Denote \eq{def rho rk1}{\rho \df { \textstyle\bigwedge^d}
\Ad : G \to \GL(V)\,,} 
and choose a nonzero element $\mathbf{p}_U$ in the 
one-dimensional subspace $\bigwedge^d \Lie(U)$ of $V$. Also 
fix a norm $\| \cdot \|$ on $V$. 
Then we have
\begin{prop}[Compactness criterion, rank one case]\name{prop: cc, rank1}
There is a finite $C \subset G$ such that the following hold:
\begin{enumerate}
\item
For any $c \in C$, $\rho(\Gamma c^{-1})\pu$ is discrete.
\item 
For any $L \subset G$, $\pi(L) \subset \ggm$ is precompact if and only
if there is $\vre>0$ such that for all $\gamma \in \Gamma$, $c \in C$, $g
\in L$
one has $\|\rho(g\gamma c^{-1})\mathbf{p}_U \| \geq \vre.$
\item
There exists $\vre_0$ such that for any $g \in G$, there is at most
one $\vv \in \rho(\Gamma C^{-1})\mathbf{p}_U$ such that $\|\rho(g)\vv\| <
\vre_0$. 
\end{enumerate}
\end{prop}

\begin{proof}
This result, essentially due to H.\,Garland 
 and M.S.\,Raghunathan, is
not usually stated in this representation-theoretic language. For the reader's
convenience we indicate here how to deduce the Proposition from
\cite{GR}. \commargin{Barak: I expanded the proof, see what 
you think. There is also a historical issue here - \cite{GR} rely essentially on
previous results of Selberg and Kazhdan-Margulis. But I have not
touched on this, from our point of view it is prehistory, almost like
debating Euclid vs. Archimedes. Dima: OK with me.}

Let $K$ be a maximal compact subgroup of $G$, let $A$ be a maximal
$\R$-split torus contained in $P$  
(which is one-dimensional in this case), let $\alpha$ be the
nontrivial character on $A$ such that $\rho(a)\pu = e^{-\alpha(a)}
\pu$, and let $$A_\eta = \{a \in A: \alpha(a) \geq \eta\}\,.$$
According to \cite[Thm. 0.6]{GR}, there is a finite
$C \subset G$, a compact subset $\Sigma \subset
U$ and $\eta \in \R$ such that 
\begin{itemize}
\item[(i)]
For any $c \in C$, $c \Gamma c^{-1} \cap U$ is a lattice in $U$.
\item[(ii)]
$
G = KA_\eta \Sigma C\Gamma \, .
$
\item[(iii)]
Given a compact $\Sigma' \supset \Sigma$, there is $s>0$ such that, if $c_1, c_2
\in C$ and $\gamma_1, \gamma_2 \in \Gamma$ are such that $KA_s \Sigma' c_1
\gamma_1  \cap KA_\eta \Sigma' c_2 \gamma_2 \neq \varnothing$ then $c_1 = c_2$ and
$\rho(c_1 \gamma_1 \gamma_2^{-1} c_2^{-1})$ fixes $\mathbf{p}_U$.
\end{itemize}

Assertion (1) follows from (i) and a result of Dani and Margulis
\cite[Thm. 3.4.11]{handbook} applied to the lattice $c \Gamma c^{-1}$. 
The direction $\implies$ in (2) is immediate from (1). To deduce
$\Longleftarrow$, note that if $g_n \in L$ is written as $g_n = k_n
a_n \sigma_n c_n \gamma_n$ as in (ii)
, and $\pi(g_n) \to \infty$ in $\ggm$, then
necessarily $\alpha(a_n) \to \infty$. Take an infinite subsequence for
which $c_n =c$ is constant and $k_n \to k_0$. Applying $\rho(g_n)$ to
$\vv_n \df \rho(\gamma_n^{-1} c^{-1})\mathbf{p}_U$ and using the fact that
$\rho(\sigma_n) \pu = \pu$ and $K$ is compact, we see that
$$\rho(g_n)\vv_n = \rho(k_na_n)\pu = e^{-\alpha(a_n)}\rho(k_n)\pu \to
0.$$
This proves (2). 

Now let $\Sigma' \supset \Sigma$ be a compact subset of $U$ which
contains a fundamental domain for the lattices $U \cap c\Gamma c^{-1}$
for each $c \in C$. Given $g \in G, c \in C, \gamma \in \Gamma$ we can use Iwasawa
decomposition to write
$$
g \gamma^{-1} c^{-1} = ka_tn'\bar{n},
$$
where $n' \in \Sigma', \bar{n} \in U \cap c\Gamma c^{-1}, k \in K$, and
$a_t \in A$. With no loss of generality assume the norm $\| \cdot \|$
on $V$ is $\rho(K)$-invariant. Since 
$$\|\rho(g \gamma^{-1} c^{-1})\mathbf{p}_U\| = 
\|\rho(ka_tn'\bar{n}) \mathbf{p}_U\| =\|\rho(ka_t) \mathbf{p}_U\| =
e^{-\alpha(t)} \| \mathbf{p}_U\|, 
$$
there exists $\vre_0$ such that if 
\eq{eq: small vector}{
\|\rho(g
 \gamma^{-1} c^{-1})\mathbf{p}_U\|
< \vre_0}
 then $t \geq s$ where $s$ is as in (iii). Writing $\bar{n} = c
\bar{\gamma} c^{-1}$ we see that $g \in K A_s \Sigma' \bar{n} c \gamma = K A_s
\Sigma' c \bar{\gamma} \gamma.$ From (iii) it follows that for a given
$g$, \equ{eq: small vector} determines $c \in C$ uniquely. Using (iii)
again, and the fact that $\rho(\bar{n})$ fixes $\mathbf{p}_U$, we
obtain (3). 
\end{proof}

\subsection{Arithmetic homogeneous spaces}
\name{arithm} Here we assume that $\Gamma$ is commensurable with the 
group of integer points of $G$, and equip $\goth{g}$ with a
$\Q$-structure which is $\Ad(\Gamma)$-invariant. Fix some rational 
basis for $\g$ and denote its $\Z$-span by  $\goth{g}_{\Z}$.
Let $P_1, \ldots, P_r$ be the maximal $\Q$-parabolic subgroups
containing a fixed minimal $\Q$-parabolic subgroup. Then it is known
\cite{Borel} that $r$ is equal to $\rank_{\Q} (G)$, i.e.\ to the dimension of a
maximal $\Q$-split torus in $G$. For $j=1, \ldots, r$ let
$\mathcal{P}_j$ denote the set of all conjugates of $P_j$ and let
$\mathcal{P} = \bigcup_{j} \mathcal{P}_j.$   
A subset $M \subset
\goth{g}$ is called {\sl horospherical\/} if it is a minimal subset
(with respect to inclusion) such that $\spa(M) = \Lie\big(\Radu(P)\big)$ for
some $P \in \mathcal{P}$.

The following is a useful criterion proved in \cite[Prop. 3.5]{with
george}.

\begin{prop}[Compactness criterion, arithmetic case]

\name{prop: cc, semisimple}
For any $L \subset G$, $\pi(L) \subset \ggm$ is precompact if and only
if there is a
neighborhood $W$ of $0$ in $\goth{g}$ such that for all $g \in L$, 
$ W
\cap \Ad(g)\goth{g}_\Z$ 
does not contain a horospherical subset.  
\end{prop}

For $j=1, \ldots, r$ let 
$U_j = \Radu(P_j),\ \goth{u}_j = \Lie(U_j), \ d_j= \dim U_j$, 
and define 
\eq{defn rho}{V_j \df
{\textstyle\bigwedge}^{d_j} \goth{g}\ \text{ and }  \ \rho_j \df  
{\textstyle\bigwedge}^{d_j} \Ad : G \to \GL(V_j)\,.}  Let $\mathcal{R}_j$ be
the 
set of all unipotent radicals of maximal
$\Q$-parabolics conjugate to $P_j$, and let $\mathcal{R} = \bigcup_j
\mathcal{R}_j$  (note that $\mathcal{R}$ contains  only $\Q$-elements,
so is a proper subset of the set of unipotent radicals of elements of
$\mathcal{P}$). For each $U \in \mathcal{R}_j$, let
$\mathbf{p}_U = \vu_1 \wedge \cdots \wedge \vu_{d_j} \in V_j$,
where $\vu_1, \ldots, \vu_{d_j} \in \goth{g}_{\Z}$ form a basis for
the $\Z$-module $\Lie(U) \cap \goth{g}_{\Z}$
($\mathbf{p}_U$ is uniquely determined up to a sign). Since
$\goth{g}_{\Z}$ is discrete, 
\eq{later}{
\inf_{U \in \mathcal{R}} \| \p_U\|>0.
}
%
Note that if $g \in G$ and $M\subset
\Ad(g)\goth{g}_\Z$ is horospherical, and we denote
$$U(M,g) \df\text{ the group with Lie algebra spanned by } \Ad(g)^{-1} M\,,$$ 
then $U(M,g) \in
\mathcal{R}$.
Moreover, for any $\vre>0$ and any $j$
there is a neighborhood $W'$ of $0$ in $\goth{g}$ such that if 
some subset of $W' \cap \Ad(g)\goth{g}_\Z$  spans $\Ad(g)\goth{u}$, where
$\goth{u} = \Lie(U), U \in \mathcal{R}_j$, then $\|\rho_j(g)\mathbf{p}_U
\| < \vre$. From this and Proposition \ref{prop: cc, semisimple} 
one finds:

\begin{cor}\name{cor: cc with reps}
Let $L\subset G$ and let $W$ be a neighborhood of $0$ in
$\goth{g}$. Suppose that there is $\vre>0$ such that for every $g \in L$ and
every horospherical $M$
contained in $W \cap \Ad(g)\goth{g}_\Z
$, 
$$\|\rho_j (g) \mathbf{p}_U \| \geq \vre,$$
where 
$U=U(M,g) \in \mathcal{R}_j.$ Then $\pi(L) \subset
\ggm$ is precompact. 
\end{cor}
We will say that a linear subspace of $\goth{g}$ is {\em unipotent} if
it is contained in the Lie algebra of a unipotent subgroup. Clearly
any subspace of a unipotent subspace is also unipotent.

\begin{prop}[\cite{with george}, Prop. 3.3]

\name{prop: algebra unipotent}
There is a neighborhood $W^{(0)}$ of $0$ in $\goth{g}$ such that for
any $g \in G$, the span of $\Ad(g)\goth{g}_\Z 
\cap W^{(0)}$ is unipotent.  
\end{prop}

We will need the following consequence of \cite[Prop. 5.3]{with george}:

\begin{prop}\name{prop: same parabolic}
Suppose for some $j \in \{1, \ldots, r\}$ that
$U = \Radu(P)$ and $U' =
\Radu(P')$, where $P,P' \in \mathcal{P}_j$. Writing $\goth{u} = \Lie(U),
\goth{u}' = \Lie(U')$, suppose further that $\spa ( \goth{u}, \goth{u}')$ is 
unipotent. Then $P= P'$ and $\goth{u} = \goth{u}'$. In
particular, for any unipotent $\goth{v} \subset \goth{g}$,   
$$\# \left\{U \in \mathcal{R}: \Lie(U)  \subset \goth{v} \right\} \leq r.
$$ 
\end{prop}

\section{Expanding subgroups}
\name{section: F expanding groups}
In this section we define $(F^+,\Gamma)$-expanding
subgroups, and give some examples. 
Unifying the notation from the two cases
considered in \S\ref{cc}, 
we define a representation $\rho: G \to \GL(V)$ as in  \equ{def rho rk1}
in the rank one case, and in
the arithmetic case we let  $$V \df \bigoplus V_j\ \text{  and }\ \rho \df \bigoplus \rho_j\,,$$
where 
$V_j,\rho_j$ are as in \equ{defn rho}.

\smallskip

We will start by citing the following
\commargin{Barak: \cite{Borel} does not contain the statement I
wrote, but the statement can be deduced from what is in that book. Is
this enough? If not is there a better statement or 
should I write a proof?} 
\begin{prop}\name{prop: Jordan} \cite[\S4,
\S8]{Borel}
For any one-parameter subgroup $F = \{g_t\}$ of $G$ 
one can
find one-parameter subgroups $\{k_t\}, \{a_t\}$ and $\{u_t\}$ of $G$,
such that the following hold:
\begin{itemize}
\item
$g_t = k_t u_t a_t$ for all $t$. 
\item
All elements in all of the above one parameter subgroups commute. 
\item
$\{\Ad(k_t): t \in \R\}$ is bounded.
\item
$\Ad(a_t)$ is semisimple and $\Ad(u_t)$ is unipotent for all
$t$. \commargin{Dima: I don't now, can rely on your judgement. But I don't think it's a good idea
to call it Jordan decomposition over $\R$, couldn't find such a term anywhere. Also decided to call it a proposition and refer to it later in \S\ref{margconj}.}
\end{itemize}\end{prop}
Note that $\Ad(a_t)$ is nontrivial if and only if $g_t$ is not quasiunipotent,
which will be our standing assumption for this section. 

Let $T$ be a maximal $\R$-split torus of $G$ containing $\{a_t\}$ and let
$\mathcal{X}(T)$ be the set of $\R$-algebraic homomorphisms $T \to \R$. 
Let $\Psi $ be the set of weights for $\rho$. 
Then we have $V =  \bigoplus_{\lambda \in \Psi} V^{\lambda}$,
where $\Psi \subset \mathcal{X}(T)$ and 
$$V^{\lambda} \df \left\{\vv \in V : \forall g \in T, \  \rho(g)\vv = e^{ \lambda(g)} \vv \right \}$$ 
is nonzero for each $\lambda \in \Psi$. 
Since $G$ is semisimple, $\Psi = -\Psi$. 
We write
\[ 
V^> \df \bigoplus_{\lambda \in \Psi, \,  \lambda(a_1) >0} V^{\lambda},\ \ 
V^0 \df \bigoplus_{\lambda \in \Psi,\,  \lambda(a_1) =0} V^{\lambda}, \ \ 
V^< \df \bigoplus_{\lambda \in \Psi.\,  \lambda(a_1) < 0} V^{\lambda}\,, 
\]
and also let 
$V^\leq \df V^0 \oplus V^<$. 
For any other representation $\tau: G \to \GL(W)$ defined over $\R$ 
we will denote by $\Psi_\tau$ the set of weights for $\tau$, and similarly will write 
$ W^{\lambda}$ for $\lambda\in \Psi_\tau$, $W^>$, $W^\leq$, etc. 

Let $\mathcal{W}$ denote the union of the $G$-orbits of the vectors
representing the appropriate Lie algebras of unipotent radicals of
parabolic subgroups. I.e., in the rank one case, let $\mathcal{W} =
\rho(G)\mathbf{p}_U$, and in the arithmetic case, $\mathcal{W} =
\bigcup_{j=1}^r \rho(G) \mathbf{p}_{U_j}.$ Note that $\mathcal{W}$ is
a compact subset of $V$. Now let us say that a subgroup $H \subset G$
is {\sl $(F^+, \Gamma)$-expanding\/} if 
\eq{eq: defn expanding}{
\rho(H)\mathbf{p} \not \subset
V^\leq \ \ \text{for any }\mathbf{p} \in \mathcal{W}\,.
}

Having introduced this definition, we can state our main theorem:

\begin{thm}\name{thm: more general} Suppose 
$G$ is a connected centerfree semisimple Lie group with no compact factors,
 $\Gamma\subset G$ is an irreducible lattice, and $F$ is a non-\qu \
one-parameter subgroup. Let $H$ be an $(F^+,\Gamma)$-expanding
subgroup of $H^+$ normalized by $F$,
 and let $\FF$ be 
as in \equ{def phi}.
Then there exists $a' > 0$ such that for any $x\in \ggm$ and $a >
a'$, the set 
\eq{efh}{\{h\in H : hx \in  E(F^+,\infty)\}}
is  $a$-winning  for the MSG induced by $\FF$.
\end{thm}

Note that \commargin{Barak:
added some notation to stress we are restricting to $H$. Dima: replaced it by an explanation
after the theorem, is it OK?}
 since   $H$ is normalized by $F$, any $\Phi_t  \in \FF$ can be viewed as a  (contracting)
automorphism of $H$, that is, the MSG considered in the above theorem is in fact induced by the restriction
of $\FF$ to $H$. 

Theorem \ref{thm: more general} will be proved in \S\ref{winning}. Meanwhile,
our goal in this section is to give some examples of $(F^+,
\Gamma)$-expanding subgroups.  We begin with the following sufficient condition:

\begin{prop}
\name{prop: november condition}
Let $A 
$ be the semisimple part of $F$ as in Proposition \ref{prop: Jordan}.
\commargin{Barak: Note that A replaces F in several
places. Made lots of small changes 
from here down to the end of the section. Dima: made lots of other small changes, please check.}
 Suppose $H$ is a unipotent group normalized by $A$, and suppose there is
a semisimple 
subgroup $L \subset G$ containing $A$ and $H$ such that:
\begin{itemize}
\item[(a)]
for any representation 
$\tau: L \to \GL(W)$ 
 defined over $\R$, if $\vw\in W$ is 
invariant under $\tau(H)$, then either $\vw$ 
is fixed by $\tau(L)$ or $\vw \notin W^\leq$;
\item[(b)]
$A$ projects nontrivially onto any simple factor of $L$;
\item[(c)]
for any $\mathbf{p} \in \mathcal{W}, \, \rho(L)$ does not fix $\mathbf{p}$. 
\end{itemize}
Then $H$ is $(F^+, \Gamma)$-expanding.
\end{prop}

\begin{proof}
Take $\mathbf{p} \in \mathcal{W}$, let 
$$\hat V  \df \spa \, \rho(H)\mathbf{p}, $$
and suppose, in contradiction to \equ{eq: defn expanding}, that $\hat V
\subset V^\leq$. Put $d = \dim( \hat V)$, consider the
representation $\tau \df \bigwedge^d \rho$ on the space $W \df
\bigwedge^d 
V$, and take a nonzero vector $ \vw \in \bigwedge^d \hat V$. 

Since $\hat V$ is
$\rho(H)$-invariant, the line spanned by $\vw$ is fixed by
$\tau(H)$. Since $H$ is unipotent, it has no multiplicative algebraic
characters, and hence $\vw$ is fixed by $\tau(H)$.  The fact that $\hat V
\subset V^\leq$ implies that $\vw \in W^{\leq}$. By property (a),
$\vw$ is $\tau(L)$-invariant, i.e.\ the subspace
$\hat V$ is $\rho(L)$-invariant. Also, since $A\subset L$, 
 $\vw$ is fixed by $\tau(A)$,
therefore $\hat V \subset V^0$. Now let 
$N$ denote the kernel of the action of
$L$ on $\hat V$. Then $N$ is a normal subgroup of $L$, and it contains $A$,
since $\hat V \subset V^0$. By (b) $N=L$, a contradiction to (c). 
\end{proof}

We remark that condition (a) in the above proposition is
equivalent to the statement that the subgroup $AH$ generated by $A$
and $H$ is {\sl epimorphic\/} in $L$, see \cite{epi}. 

\smallskip

We are now ready to exhibit subgroups which are $(F^+,\Gamma)$-expanding.

\begin{prop}\name{prop: horospherical F expanding}
Let $G$ and $\Gamma$ be as in Theorem \ref{thm: C}, and let $F$ be non-quasiunipotent.
Then the expanding horospherical subgroup $H^+$ defined in \equ{defn of
U+} is $(F^+, \Gamma)$-expanding. 
\end{prop}

Note that this proposition, together with Theorem \ref{thm: more general}, 
immediately implies Theorem \ref{thm: C}. 
For the proof of Proposition  \ref{prop: horospherical F expanding} we will need the following fact, which will be used later on as well:
\begin{prop}\name{prop: since Gamma irreducible}
Suppose $G$ is a semisimple centerfree Lie group with no compact factors,
$\Gamma$ is an irreducible lattice, and $L$ is a 
nontrivial normal subgroup of $G$. Then the $L$-action on $\ggm$ is
minimal and uniquely ergodic.
\end{prop}

\begin{proof}
The fact that $L\Gamma$ is dense in $G$ follows immediately from the
irreducibility of $\Gamma$. Since $L$ is normal, so
is $gL\Gamma = Lg\Gamma$ for any $g \in G$, so the action is
minimal. If $\mu$ is any $L$-invariant probability measure on
$\ggm$, then it lifts to a Radon measure $\hat{\mu}$ on $G$ which is
left-invariant by $L$ and right-invariant by $\Gamma$. Since $L$
is normal in $G$, this means that $\hat{\mu}$ is right-invariant by
$L\Gamma$. Since the stabilizer of a Radon measure is closed and
$L\Gamma$ is dense, this means that $\hat{\mu}$ is $G$-invariant, i.e.\ it
is a Haar measure on $G$. In particular $\mu$ is the unique
$G$-invariant probability measure.
\end{proof}

\ignore{
We also need to discuss the Lie-algebraic description of $H^+$. \commargin{Moved $\Phi$ here}
Namely,  let $\Phi$ be the set of roots, i.e. weights
of $\Ad: G \to \GL(\goth{g})$, and for $\alpha\in\Phi$ let $\g_\alpha$ be the
root subspace of $\g$. \commargin{Again, how is $\g_\alpha$ exactly defined? don't we
need to complexify?}
Now define 
$$\goth{h}^+ \df \bigoplus_{\mathrm{Re} \, \alpha(a_1) >
0}\goth{g}_\alpha, \ \ \goth{h}^0 \df \bigoplus_{\mathrm{Re} \, \alpha(a_1) =
0}\goth{g}_\alpha,\ \ \goth{h}^- \df \bigoplus_{\mathrm{Re} \, \alpha(a_1) <
0}\goth{g}_\alpha\,.
$$
Then $\h^+$ is exactly the Lie algebra of $H^+$ as in \equ{defn of U+}.
The fact that $F$ is not \qu\ 
implies that $\h^+ \ne \{0\}$ and $\h^- \ne \{0\}$.
Moreover one can consider   subgroups  $H^0$, $H^-$ of $G$ corresponding to $\h^0$
and $\h^-$; note that $H^-$ is the \ehs\ corresponding to $g_{-1}$. 
Since \eq{decomp}{\g = \h^+ \oplus\h^0\oplus\h^-\,,} it follows that  the multiplication map 
\eq{mult}{(h^-,h^+,h^0)\mapsto
h^-h^+h^0,
\ \ H^-\times H^+ \times H^0\to G \,,}
 is a homeomorphism in the neighborgood of identity
(regardless of the order of the factors).}

\begin{proof}[Proof of Proposition \ref{prop: horospherical F expanding}]
We apply Proposition \ref{prop: november condition} with
$L$ being the smallest normal subgroup of $G$ which contains $A$. We need
to show that $H^+$ is contained in $L$ and normalized by $A$, and verify
conditions (a), (b) 
and (c). 

Since the projection of $H^+$ onto $G/L$ is expanded by conjugation under the
projection of $A$, which is trivial, $H^+$ is contained in $L$. 
Note that if we consider the representation $\Ad: G \to
\GL(\goth{g})$, where $\goth{g}$ is the Lie algebra of $G$, 
then, using the notation above, the Lie algebra $\goth{h}$ of
$H^+$ is precisely $\goth{g}^>$. In particular, it is
$\Ad(A)$-invariant, so $H$ is normalized by $A$.

Condition (b) is
immediate from the definition of $L$. To uphold (c), suppose that
$\rho(L)$ fixes $\mathbf{p}\in\mathcal{W}$. There is some $g \in G$ so that
$\mathbf{p} = \rho(g)\pu$, where $\pu$ is a vector representing a
unipotent radical of a parabolic subgroup. 
So $\rho(g^{-1}Lg)=\rho(L)$ fixes $\pu$. On the other
hand, by Proposition \ref{prop: since Gamma irreducible} we have 
$$\rho(G)\pu \subset \cl{\rho(\Gamma L)\pu}
= \cl{\rho(\Gamma)\pu}\,.$$
Using either Proposition
\ref{prop: cc, rank1}(1) in the rank-one case, or rationality of $\rho$ and
$\mathbf{p}_U$ in the arithmetic case, we see that after replacing
$U$ with a conjugate we may assume that $\rho(\Gamma)\pu$ is
discrete. Therefore  $\rho(G)$ fixes $\pu$, thus $\Lie(U)$
is $\Ad(G)$-invariant. Hence $U$ is a normal unipotent subgroup of $G$,
which is impossible.  

It remains to verify (a). Let $\tau: L \to \GL(W)$ be a representation,
let $\goth{l} = \Lie(L)$, and let $\Phi$ be the set 
of roots, i.e. weights 
of $\Ad: L \to \GL(\goth{l})$. 
Recall that 
\eq{eq: derivative}{d\tau(\goth{l}_\alpha)W^{\lambda} \subset
W^{\lambda+\alpha,}} 
where $\lambda \in \Psi_\tau, \, \alpha \in \Phi,\,
\goth{l}_\alpha$ is the root subspace corresponding to the root $\alpha$ of
$L$, and $d\tau$ is the derivative map $\goth{l} \to 
\End(W)$. Let $Q$ be the parabolic subgroup of $L$ with Lie algebra
$\goth{q} = \goth{l}^{\leq}$, 
then $\goth{q} \oplus \goth{h} = \goth{l}$ (as vector spaces), hence
$QH^+$ is Zariski dense in $L$. Also by \equ{eq: derivative},
$W^\leq$ is $\tau(Q)$-invariant. 
Now suppose $\vw \in W$ is a vector
 fixed by $\tau(H^+)$ and contained in $W^\leq$. Then $\tau(QH^+)\vw \subset
W^\leq$. Therefore $\tau(L)\vw $ is a subset of $ W^\leq$, hence so is $\hat V
\df \spa \, \tau(L)\vw$, a $\tau(L)$-invariant subspace. Since $L$ is
semisimple, $\Phi 
= -\Phi$, which implies that $\hat V \subset W^0$. In particular $L \mapsto
\End(\hat V)$ has a kernel $N$ which contains $A$. Since $A$ projects
nontrivially onto every factor of $L$ we must have $L \subset N$, i.e. $L$
acts trivially on $\hat V$, in particular $\vw$ is $\tau(L)$-invariant. 
\end{proof}




As described in the introduction, Dani \cite{Dani-rk1} proved that
$E(F, \infty)$ is thick whenever $G$ is a rank one semisimple Lie
group and $F$ 
is non-\qu. He did this by
considering the Schmidt game 
played on the group $H$ corresponding to a certain subgroup of $H^+$.
It turns out that Dani's result 
is a special case of
Theorem \ref{thm: more general}:

\begin{prop}
Suppose $\rank_{\R}(G)=1.$ Then any 
nontrivial connected subgroup $H$ of
$H^+$ which is normalized by $F$ is $(F^+,\Gamma)$-expanding. 
\end{prop}

\begin{proof}
Clearly any group containing an $(F^+, \Gamma)$-expanding subgroup
is also $(F^+, \Gamma)$-expanding; therefore it suffices to 
prove the proposition when $\dim(H) = 1$. 
By the Jacobson-Morozov lemma one can find a one-dimensional subgroup
$H^-$ which is opposite to $H$ in the following sense: the group $L$
generated by $H$ and $H^-$ is locally isomorphic \commargin{Barak:
It is hard to find a good
reference to the Jacobson Morozov lemma but I think it is on the level
of algebras so need to discuss local isomorphism, which is sometimes
also called isogeny.} to $\SL_2(\R)$, and its diagonal
subgroup $F'$, which normalizes both $H$ and $H^-$, centralizes $F$. Since
$\rank_{\R}(G)=1$, we must have $F=F'$.  
%
%
In particular $L$ is generated by unipotent elements, and contains $FH$ as a
Borel subgroup.
The
representation theory of $\goth{sl}_2(\R)$ now shows \commargin{to be on the
safe side, even if this is not $\SL(2,\R)$ it has the same properties
by local isomorphism.} that condition 
(a) of  Proposition \ref{prop: november condition} holds, and (b) is
immediate. 
%
It remains to check (c). Suppose that $\rho(L)$ fixes $\mathbf{p}$,
that is $L$ normalizes a 
parabolic $P$. Since it is simple, it fixes its volume element, hence
so does $F$. Since the normalizer of a parabolic has elements of $G$
which do not preserve its volume 
element, we find that the normalizer of $P$ contains a torus of rank
$2$, contradicting the fact that $\rank_{\R} (G) =1$.
\end{proof}

We conclude the section with a higher rank example
related to \da.
 Let $m,n\in \N$ be positive integers,  put $k = m+n$,  and consider 
 $G = \SL_k({\R})$, $\Gamma = \SL_k({\Z})$, and 
$F = \{g_t\}$, where
$$
g_t =
\text{diag}(e^{t/m},\dots , e^{t/m}, e^{- t/n},\dots , e^{-
t/n})
$$
(see case  \equ{case1} discussed in the introduction). Clearly the group 
\eq{group di}{H = \{u_Y : Y \in M_{m,n}\},} 
  where
 \eq{def uy}{
u_Y \stackrel{\mathrm{def}}{=} \left(
\begin{array}{ccccc}
I_{m} & Y \\ 0 & I_{n} 
\end{array}
\right) 
}
 (here and hereafter $\mr$ is the space of $m\times n$ matrices with real entries
 and $I_{\ell}$ stands for the $\ell\times \ell$ identity matrix), is expanding horospherical with respect to $g_1$, hence $(F^+, \Gamma)$-expanding
by Proposition  \ref{prop: horospherical F expanding}. Now, 
more generally, consider
 $F_{\vr,\vs} \df
 \{g^{(\vr,\vs)}_t\}$ where
\eq{new defn gt}{
g^{(\vr,\vs)}_t =  \diag(e^{r_1t}, \ldots,
e^{r_{m} t}, e^{-s_1t}, \ldots,
e^{-s_{n}t} )\,,
}
with 
\eq{def rs}{r_i,s_j > 0\quad\text{and}\quad\sum_{i=1}^m r_i = 1 =
 \sum_{j=1}^n s_j\,.}
Then $H$ is contained in the expanding horospherical  subgroup with respect to $g^{(\vr,\vs)}_1$,
and the containment is proper when some components of either $\vr$ or $\vs$ are different.
Nevertheless, the following holds:

\begin{prop}\name{prop: epimorphic}
The group \equ{group di} is $(F_{\vr,\vs}^+, \Gamma)$-expanding for any $\vr,\vs$ as in \equ{def rs}.  
\end{prop}

\begin{proof}
One can use Proposition \ref{prop: november condition} with $L=G$. Conditions (b) and (c)
are immediate, and (a) follows from the proof of \cite[Lemma 2.3]{di}. 
\end{proof}

See Corollary \ref{cor: E} for a \di\ application of the above proposition.

\section{Bounded trajectories are winning}
\name{winning}
In this section we will prove Theorem \ref{thm: more general};
throughout this section the assumptions will be as in the theorem.
We will need the following lemmas.

\ignore{
\begin{lem}\name{lem: finite covers}
Suppose $\Gamma' \subset \Gamma$ is also a lattice in $G$, $p: G/\Gamma'
\to \ggm$ the natural projection, and $\bar{x} \in G/\Gamma', \,
x=p(\bar{x}) \in \ggm$. Denote by $E'(F^+,\infty)$ the set of points
$z \in G/\Gamma'$ for which $\cl{F^+z}$ is bounded. Then $\{h\in H :
hx \in  E(F^+,\infty)\}$  is a winning set for the MSG induced by $\FF
$ if and only if so is $\{h\in H : h\bar{x} \in  E(F^+,\infty)\}$. 
\end{lem}

\begin{proof}
This follows immediately from two facts: $X' \subset G/\Gamma'$ is
bounded if and only if $p(X)$ is bounded; $p$ intertwines the
$H$-action on $G/\Gamma'$ and $\ggm$ respectively. 
\end{proof}
}

\begin{lem}
\name{lem: polynomial maps}
Let $\| \cdot \|$ denote a norm on $\R^n$, and also, for a
polynomial $\varphi: \R^{\ell} \to \R^n$, let $\|\varphi\|$ denote the maximal
absolute value of a coefficient of $\varphi$.  
Given natural numbers $r, \ell, n, d$ and a ball 
$B' \subset
\R^{\ell}$, 
there are positive constants $\eta, c$
such that the following holds. Suppose for $j=1, \ldots, r$ that $\varphi_j:
\R^{\ell} \to \R^{n_j}$ is a polynomial map in $\ell$ variables of
degree at most $d$, with $n_j \leq n$.
Then there is a ball 
$B \subset B'$ of diameter $c$ such that for $j=1, \ldots, r$, 
$$\x \in B \ \implies \ \|\varphi_j(\x)\| \geq \eta  \|\varphi_j\|.$$
\end{lem}

\begin{proof}
By induction we may assume that $r=1$, i.e. that there is only one
polynomial map, which we denote by $\varphi$. It is easy to see that the general
case follows from the case $n=1$. 

Note that
the space of all possible $\varphi$'s is a finite dimensional real vector
space. Since the two norms $\|\varphi\|$ and $\varphi \mapsto \max_{\x \in B'}
|\varphi(\x)|$ are equivalent, there is $\eta'>0$ (independent of $\varphi$) such
that $B'$ contains a point $\x'$ such that $|\varphi(\x')| \geq \eta' \|\varphi\|.$
By continuity and compactness, there is $C>0$
such that for any $\varphi$, $\|\nabla \varphi \| \leq C \|\varphi\|$. Now take $\eta =
\eta'/2$ and $c = \min\big(
\eta' /2C,\diam(B')\big)$. On any ball of diameter $c$ the variation of $\varphi$ is
at most $\eta' \|\varphi\|/2$. Thus we may take $B$ to be a sub-ball of $B'$
of diameter $c$ containing $\x'$. 
\end{proof}

In our application the polynomial maps will come from the
representation $\rho$ defined at the beginning of $\S 4$. 
Given $\vv \in V$,
define 
$$\varphi = \varphi^\vv: H \to \R^{n}, \ \ \ \varphi(h) \df \rho\left ( h \right)\vv\,.$$  
We will think of $\varphi$ as a polynomial map in the following sense. Since
$H$ is unipotent, the exponential map $\goth{h} \df \Lie(H) \to H$ is a 
polynomial. Fix a basis $X_1, \ldots, X_{\ell}$
for $\goth{h}$, where $\ell =
\dim H$, and for $\varphi$ as above consider
$$\til \varphi : \R^{\ell} \to \R^{n}, \ \ \til \varphi(s_1, \ldots, s_{\ell}) =
\varphi\left(\exp\left(\sum^{\ell} s_iX_i\right) \right). $$
When discussing $\varphi$ as a polynomial (e.g. when discussing the degree
or coefficients of $\varphi$) we will actually mean $\til \varphi$. Since $\rho
\circ \exp = \exp \circ \,d\rho$, the
degree of $\varphi$ is no more than $\dim (V) -1$. Let $\Pi: V \to V^{>}$ be the projection
corresponding to the direct sum decomposition $V = V^\leq \oplus
V^{>}$. One can fix the norm $\| \cdot \|$ on $V$ and choose 
$a_0 > 0$ so that
\eq{norm requirements}{\vv \in V,\ t\geq a_0 \ \ \implies \ \ \|\vv\| \geq
\| \Pi \vv\|, \ \|\rho(g_t) \Pi \vv \| \geq \|\Pi \vv\|.
}

\begin{lem}
\name{lem: polys nonzero}
For any $\vre_1>0$ there is
$\vre_2>0$ such that if $\vv 
\in \mathcal{W}, \ \|\vv\| \geq
\vre_1$, then at least one of the coefficients of $\Pi \circ \varphi^{\vv}$ is
greater 
than $\vre_2$. 

\end{lem}

\begin{proof}
Since the set 
$\{\vv  \in \mathcal{W}: \|\vv\| \le \vre_1\}$
is compact, it suffices to show that the linear map $\vv \mapsto \Pi
\circ \varphi^\vv$
is injective; but unravelling definitions, one sees that this is
precisely equivalent to \equ{eq: defn expanding}. 
\end{proof}

\begin{proof}[Proof of Theorem \ref{thm: more general}]
We are given an admissible $D_0$. 
Our goal is to find $\A > 0$ such that for any $g\in G$, the set \equ{efh}
is $\A$-winning for the $\{\cd_t\}$-game with $\{\cd_t\}$ as in \equ{def dt},
where $x = \pi(g)$. Let $B' \subset D_0$ be a ball, let $n = \dim V$, let $d=n-1$,
and take $r = 1$ if $\rank_\R(G) = 1$, and $r = \rank_\Q(G)$ otherwise.  
Then  let $c, \eta$ be the constants as in Lemma \ref{lem: polynomial maps} corresponding to
$r
,\ell, n,d$ and $B'$. Replacing ${g}$ with $h{g}$
for an appropriate $h \in H$ we may assume that $e \in B'$. 
Since the maps $\Phi_{t}$ are contracting,
we may choose $a' > a_0$, with $a_0$ as in \equ{norm requirements}, large enough so
that for any $\A > a'$, any ball in $B'$ of diameter $c$ contains a right-translate of
$\Phi_{\A}(D_0)$. 

Now given any choice of $\B>0$, an initial $t_0 > 0$, $x=\pi(g)\in\ggm$
and any $B_1$ which is a right-translate of $\Phi_{t_0}(D_0)$, we need to
describe a compact $K \subset \ggm$ and a strategy for Alice, such
that for the choices of sets $B_1, B_2, \ldots$ by Bob, there are
choices of sets $A_{i} \subset B_{i}$ for Alice with
\eq{need to show rk1}{
h \in A_i,\, i \in \N \ \ \implies \ \ g_ih{x} \in
K,
}
where 
\eq{def gi}{g_i \df g_{t_0+i(\A+\B)}\,.}
This will ensure that the trajectory $F^+h_{\infty}{x}$ is bounded,
where $h_{\infty}$ is the intersection point \equ{int}. The
definition of $K$ and the description of the strategy will now be given
separately for the rank-one and arithmetic cases.

\medskip

\noindent{\bf The rank-one case.} Assume the notation of
\S\ref{rk1}. By boundedness of $B_1$ and
Proposition \ref{prop: cc, rank1}(2) there is $\delta_1>0$ so that 
\eq{eq: property of delta1}{
c \in C, \, h \in B_1, \, \gamma \in \Gamma \  \ \implies \ \
\|\rho(g_1h{g}\gamma c^{-1})\mathbf{p}_U\| \geq \delta_1.}
Let $\vre_0$ be as in Proposition \ref{prop: cc, rank1}(3). Since
$D_0$ is bounded, by making $\delta_1$ smaller we can assume that 
$\delta_1<\vre_0$ and 
\eq{eq: property of delta2}{
h \in g_{\A}D_0D_0^{-1} g^{-1}_{\A},\ g' \in G,\ \vv \in V,\  \|\rho(hg')\vv \| < \delta_1 \implies
\|\rho(g')\vv \| < \vre_0. 
}
Using Lemma \ref{lem: polys nonzero}, choose $\delta_2$ so that 
\eq{choice of delta3}{
\vv  \in \mathcal{W}, \, \|\vv\| \geq \delta_1 \implies \|\Pi
\circ \varphi^{\bar{\vv}} \| \geq \delta_2, \ \ \mathrm{where} \ \bar{\vv} = g_{\B}\vv.
}
Now let 
$$\delta_3 = \min \{\eta \delta_2, \delta_1\},$$
where $\eta$ is as in Lemma \ref{lem: polynomial maps}, and 
$$K =
\pi\left(\left\{g' \in G: \forall \gamma \in \Gamma, c \in C,
\|\rho(g' \gamma c^{-1})\mathbf{p}_U\| \geq \delta_3 \right\}\right),$$
which is compact in light of Proposition \ref{prop: cc, rank1}(2). 

We claim Alice can make moves so that for each $i \in \N$, each $h \in
A_i$, each $c \in C$, and each $\gamma \in \Gamma$, at least one of
the following holds:  
\begin{itemize}
\item[(a)]
$\|\rho(g_ih{g} \gamma c^{-1}) \mathbf{p}_U \| \geq \delta_1$;
\item[(b)]
$\|\Pi \circ \rho \left
(g_ih{g} \gamma c^{-1}\right ) \mathbf{p}_U
\| \geq \delta_3$. 
\end{itemize}

Clearly this will imply \equ{need to show rk1} and conclude the proof. 

\smallskip
The claim is proved by induction. From \equ{eq: property of delta1} it
follows that (a) holds for $i=1$ and for all $h \in B_1$, so Alice can
choose $A_1$ at her whim. Assume the claim is true for $i-1$, where $i
\geq 2$. We claim that (a) can fail for $i$ with at most one vector. That
is, there is at most one vector $\vv =
\rho(\gamma c^{-1})\mathbf{p}_U$ for which there is $h \in B_{i-1}$ with 
$\|\rho(g_ih {g})\vv\| < \delta_1$. For suppose, for $\ell = 1,2$, that
$\vv_\ell=\rho(\gamma_\ell c_{\ell}^{-1})\mathbf{p}_U$ are two such vectors, with
$h_\ell \in B_{i-1}$ the corresponding elements. Then the $h_\ell$ are
both in $B_{i-1}$ which is a right translate of
$g_{\B}^{-1}g_{i-1}^{-1}D_0 g_{i-1}g_{\B}$, i.e. 
$$g_{\B} g_{i-1}h_1h_2^{-1}g_{i-1}^{-1} g_{\B}^{-1} \in D_0 D_0^{-1}.
$$
Therefore $g_ih_1 = h'g_ih_2$ where $h' = g_ih_1h_2^{-1}g_i^{-1} \in
g_{\A} D_0D_0^{-1} g_{\A}^{-1}$. In view of \equ{eq: property
of delta2} and Proposition \ref{prop: cc, rank1}(3) we must have $\vv_1 =
\vv_2$. 

So no matter how Alice chooses $A_i$, (a) will hold except
possibly for one vector $\vv = \rho(\gamma c^{-1})\mathbf{p}_U \in \mathcal{W}.$
By the induction hypothesis, either $\|\Pi \circ \rho(g_{i-1}h{g})\vv\|
\geq \delta_3$ or $\|\rho(g_{i-1}h{g})\vv\|\geq \delta_1$. If the
former occurs, it also occurs for $i$ because of \equ{norm
requirements}, implying (b). If the latter occurs,
write 
$$
B_{i-1} = g_{\B}^{-1}g_{i-1}^{-1}D_0g_{i-1}g_{\B}h_0, \ \ \mathrm{where} \
h_0 \in H\,,
$$
and set
$$\vv' = \rho(g_{i-1} h_0 {g})\vv \in \mathcal{W}, \ \ \bar{\vv} = \rho(g_{\B})\vv'\,.
$$
Then $\|\vv'\| \geq \delta_1$ so by \equ{choice of
delta3} we have $\|\Pi \circ \varphi^{\bar{\vv}} \| \geq \delta_2.$ 
By Lemma \ref{lem: polynomial maps} we can choose a ball $B \subset
B'$ of diameter $c$ so that 
\eq{choice of ball}{
h \in B \implies \|\Pi \circ \varphi^{\bar{\vv}}(h)\| \geq \eta \delta_2 \geq \delta_3.}
By the choice of $\A$, 
$$D_0 \supset B' \supset B \supset g_{\A}^{-1} D_0
g_{\A} h_1$$ 
for some $h_1 \in H$. Since 
$$B_{i-1}  \supset
 g^{-1}_{i-1}g^{-1}_{\B}B'g_{\B}g_{i-1} h_0,$$
Alice can choose 
$$A_i =   (g_\B g_{i-1})^{-1}g^{-1}_{\A} D_0 g_{\A} h_1
g_{\B}g_{i-1} h_0 = g_i^{-1} D_0 g_ih_2\,,$$
where $h_2 = g_i^{-1}g_\A h_1 g_\A^{-1}g_i h_0\in H$.
Now for all $$h = (g_\B g_{i-1})^{-1}h'g_\B g_{i-1}h_0 \in
A_i\,,$$ where $h'\in g_{\A}^{-1} D_0g_{\A} h_1 \subset B$, we have  
\[
\begin{split}
\|\Pi \circ \rho(g_ih {g} c \gamma) \mathbf{p}_U \| & \geq \|\Pi \circ
\rho(g_{\B} g_{i-1} h {g})\vv \| \\ 
& = \|\Pi \circ \rho(g_{\B}g_{i-1}h h_0^{-1}g_{\B}^{-1}g_{i-1}^{-1})\bar{\vv} \| \\
& = \|\Pi \circ \rho (h') \bar{\vv}\| \under{\equ{choice of ball}}\geq \delta_3\,,
\end{split}
\]
as required. 

\medskip

\noindent{\bf The arithmetic case.} 
Let $W^{(0)}$ be a neighborhood of $0$ in $\goth{g}$ as in Proposition
\ref{prop: algebra unipotent}. Let $W^{(1)} \subset W^{(0)}$ be a
neighborhood of $0$ in $\goth{g}$ such that 
\eq{choice of W1}{
h \in D_0D_0^{-1}, \, \vv \in W^{(1)} \ \implies \ \Ad(h)\vv \in
W^{(0)}.
}
Recall that we are given $x = \pi(g)\in\ggm$, and that the choice of $t_0 > 0$, $\B > 0$ and $\A > a'$
has defined the sequence $\{g_i\}$ as in \equ{def gi}.
Since $B$ is compact and in 
view of \equ{later},
there is $\vre_0>0$ such that 
$$
U \in \mathcal{R}
, \, h \in
B_1 \  \ \implies \ \ \|\rho
(g_1h{g})\mathbf{p}_U\|
\geq \vre_0\,.$$ 

Choose $\vre_1$ so that 
\eq{choice of epsilon1}{
\|\vv \| \geq \vre_0, \ h \in D_0D_0^{-1}, 
\ \implies \ \ \|\rho
(
h) \vv \| \geq \vre_1.
}
Let $\vre_2$ be the constant corresponding to $\vre_1$ as in Lemma
\ref{lem: polys nonzero}, let
$$\vre_3 = \min \{\eta \vre_2, \vre_0\}$$
and define  $K \subset \ggm$ to be the set of $x = \pi(g')$ such that whenever $
 M \subset \Ad(g')\goth{g}_\Z \,\cap\, W^{(1)}$ is horospherical, one has
$\|\rho
(g')\mathbf{p}_U\| \geq \vre_3$, where $U=U(M,g') 
\in \mathcal{R}
$. 
Then $K$ is compact by Corollary \ref{cor: cc with reps}. 
\smallskip

We will show how Alice 
can make moves so that for each $i \in \N$, each $h \in A_i$ 
and each
$U \in \mathcal{R}
$ for which there exists a horospherical
subset $M \subset \Ad(g_ihg)\goth{g}_{\Z} \cap W^{(1)}$ with
$U =U(M, g_ih{g})$,
at least one of the following holds: 
\begin{itemize}
\item[(i)]
$\|\rho
(g_ih{g}) \mathbf{p}_U \| \geq \vre_0$;
\item[(ii)]
$\|\Pi
\circ \rho
 \left
(g_ih{g} \right ) \mathbf{p}_U
\| \geq \vre_3$. 
\end{itemize}

Clearly this will imply \equ{need to show rk1} and conclude the proof. 

\medskip

The choice of $\vre_0$ implies that (i) holds for $i=1$, for all $h
\in B_1$. Thus Alice can make an arbitrary choice of $A_1$. We continue
inductively, assuming the statement for $i-1$, where $i \geq 2.$ First
suppose (ii) holds for $i-1$. We have 
$g_i = g_{\A+\B}g_{i-1}$ and by \equ{norm requirements},
$$\|\Pi
 \circ \rho
  \left(g_ih{g} \right ) \mathbf{p}_U
\| = \|\rho
(g_{\A+\B}) \circ \Pi
 \circ \rho
  \left
(g_{i-1}h{g} \right ) \mathbf{p}_U \| \geq 
\vre_3.
$$
Now consider the set $\mathcal{U}$ of
$U \in \mathcal{R}$ for 
which there exist $h=h\left(U\right) \in B_{i-1}$ and a horospherical subset $M
=M(U)\subset \goth{g}_{\Z}$ such that 
$$\Ad(g_{i-1}h{g})M \subset
W^{(1)},$$
and such that (ii) does not hold for $\mathbf{p}_U$ and 
$h$. Note that by induction hypothesis, after replacing $i$ with
$i-1$, statement (i) does hold for $\mathbf{p}_U$ and $h$.

We claim that 
\eq{number bounded}{
\# \mathcal{U} \leq r.
} 
To see this, in view of Proposition \ref{prop: same parabolic}, it suffices to show that 
the subalgebra $\goth{v}$ spanned by all $M(U), \, U
\in \mathcal{U}$, is unipotent. 
Fix $h_0 \in B_{i-1}$. It
will be convenient for later calculations to take for $h_0$ the
appropriate right-translate of $e \in B'$, that is 
\eq{convenient}{
B' \subset g_{i-1} (B_{i-1}h_0^{-1}) g_{i-1}^{-1}.
}
Suppose $U
\in \mathcal{U}, \, h=h\left(U\right), \, M = M\left(U\right).$
Then $B_{i-1}$ is a right-translate of $g_{i-1}^{-1}D_0g_{i-1}$, and hence 
$$g_{i-1}h_0h^{-1}g_{i-1}^{-1} \in g_{i-1}(B_{i-1}B_{i-1}^{-1})g_{i-1}^{-1} = D_0D_0^{-1}\,;$$
i.e., for $M' = \Ad(g_{i-1}h{g})M \subset W^{(1)}$ 
in view of  \equ{choice of W1} one has
$$
\Ad \left(g_{i-1}h_0{g}\right )M = \Ad \left(g_{i-1}h_0h^{-1}g_{i-1}^{-1} 
\right) M' \subset W^{(0)}.
$$

By the choice of $W^{(0)}$, this implies that the subalgebra spanned by 
$$\bigcup_{U \in \mathcal{U}} \Ad\left( g_{i-1} h_0 {g}
\right)M\left(U\right)$$
is unipotent, hence so is $\goth{v}$. This proves \equ{number bounded}.
 
\medskip

For any $U \in \mathcal{U}$
and $h=h\left(U\right) \in B_{i-1}$ it follows  that 
$$\|\rho
(g_{i-1} h {g}) \mathbf{p}_U \| \geq \vre_0\,.$$
Hence, using \equ{choice of epsilon1}, we deduce that
$$\|\vv
\| \geq \vre_1, \ \ \mathrm{where} \ \vv = \vv(U) = 
\rho
\left(g_{i-1} h_0 {g} \right) \mathbf{p}_U.$$

Consider the polynomials $h\mapsto \Pi
\circ \varphi
^{\vv(U)}(h)$ for all $U \in \mathcal{U}$. By
Lemma \ref{lem: polys nonzero} 
each has a coefficient of at least $\vre_2$ in absolute value. In view of Lemma
\ref{lem:
polynomial maps}, there is a ball $B \subset B'$ of diameter
$c$ such that $\|\Pi
\circ \varphi
^{\vv(U)}(h)\| \geq \vre_3$ for all $h
\in B$ and 
all $U \in \mathcal{U}$. By the choice of $\A$, $B$ contains a
translate $\til D$ of $g_{\A}^{-1} D_0 g_{\A}$. Now let $A_i=g_{i-1}^{-1} \til D g_{i-1}h_0$ be a move
for Alice. This is an allowed move since $A_i$ is a right-translate of
$\Phi_{t_0+(i-1)(\A+\B)+\A}(D_0)$ and since, by \equ{convenient},
$A_i \subset B_i.$ 
For each $h \in A_i$ let $h' = g_{i-1}hh_0^{-1}g_{i-1}^{-1}$. Then $h'
\in B$ and therefore 
$$
\varphi
^\vv(h')  = \rho
 \left( g_{i-1}hh_0^{-1}g_{i-1}^{-1}
\right)\rho
\left(g_{i-1} h_0 {g} \right) \mathbf{p}_U = \rho
\left(g_{i-1}h {g} \right) 
\mathbf{p}_U.
$$
Thus $\|\Pi
\circ \rho
\left(g_{i-1}h {g} \right) 
\mathbf{p}_U \| \geq \vre_3$, and, using \equ{norm requirements}, we obtain (ii).  
\end{proof}

\section{Escaping points: proof of Theorem \ref{thm: B}} 
\name{esc points}
In this and the next sections we will analyze the local geometry of
$\ggm$. 
In what follows,  $B(x,r)$ stands for the open ball centered at $x$ of
radius $r$. To avoid confusion, we will sometimes put subscripts
indicating the underlying metric space. 
If the metric space is a group and $e$ is its identity element, 
we will simply write $B({r})$ instead of $B(e,{r})$. For a subset $K$ of a metric space,
$K^{(\vre)}$ will denote the $\vre$-neighborhood of $K$.
\smallskip

Let $G$ be a real Lie group, $\g$ its Lie algebra, let $F$ as in \equ{f}
be non-quasiunipotent, and denote the Lie algebra of $F$ by
$\goth{f}$. Using the notation introduced in
\S\ref{section: F expanding groups}, 
let $\goth{g}^>$ 
be the expanding subalgebra corresponding to the representation
$\Ad: G \to \GL(\goth{g})$.  
\ignore{
Recall that one can define \commargin{Apparently some of this is needed
in the expanding subgroups section, so this has to be done in a uniform way
and moved to an earlier section}  
$\h^+$, $\h^0$, $\h^-$ to be the subalgebras of $\g$ with complexifications
$$
\h^+_{{\Bbb C}} = \text{span}(E_\lambda : |\lambda|>1),\,
\h_{{\Bbb C}}^0 = 
\text{span}(E_\lambda : |\lambda|=1),\, \h_{{\Bbb C}}^- =
\text{span}(E_\lambda : |\lambda|<1)\,,
$$
where $E_\lambda$ is the generalized eigenspace of Ad$\,g_1$ in the complexification of $\g$.}
Then $\goth{g}^>$ is exactly the Lie algebra of $H^+$ as in \equ{defn of U+}.
The fact that $F$ is not \qu\ 
implies that $\goth{g}^>$ is nontrivial. 
Let us fix a Euclidean structure on $\g$ so that $\goth{f}$ and
$\goth{g}^>$ are 
mutually orthogonal subspaces, and
transport it via right-multiplication to define
a right-invariant Riemannian metric on $G$. 
This way, the exponential map $\exp:\g\to G$ becomes almost an isometry locally around $0\in\g$.
More precisely, for any positive $\vre$ one can choose a neighborhood $U$ of $0\in\g$
such that $\exp|_U$ is $(1+\vre)$-bilipschitz.
 
If $\Gamma$ is a discrete subgroup of $G$, we will  equip
 $\ggm$ with the Riemannian metric coming from $G$. 
Then for any $x\in\ggm$ the orbit map
 $G\to\ggm$, $g\mapsto gx$ is an isometry when restricted to 
a sufficiently small (depending on $x$) neighborhood of $e\in G$. 

\smallskip
Our main goal in this section is to prove the following generalization of Theorem \ref{thm: B}:

\begin{thm}\name{thm: B more general}
Let $G$ be a Lie group, $\Gamma\subset G$ discrete, $F$ as in \equ{f},
$H^+$ as in \equ{defn of U+}. Let $H$ be a nontrivial connected
subgroup of $H^+$ normalized by $F$, and let 
$\FF$ be as in \equ{def phi} to $H$. \commargin{Barak:
Same notation as the one I introduced in Theorem \ref{thm: more general}. Dima: hope the
remark I made there is sufficient. Also, note that I removed the separate formal proof of Theorem \ref{thm: B}, 
hope it's OK.} 
Then there exists $a'' > 0$ such that for any $\A > a''$ and any $x,z\in\ggm$, the set
\eq{esc discrete}{\big\{h\in H : hx \in  E(F^+,z)\big\}}
is an $\A$-winning set for the MSG induced by $\FF$.
\end{thm}

\begin{proof}
Recall that we are given an admissible $D_0 \subset H$, and our
goal is to find $\A$ such that the set \equ{esc discrete}
is $\A$-winning for the $\{\cd_t\}$-game with $\{\cd_t\}$ as in \equ{def dt}. 
From the admissibility of $D_0$ and contracting properties of $\{\Phi_t\}$ it 
follows that there exist positive $a''$ and $\vre_0$ such that
for any $\A > a''$,
\eq{two choices}{\begin{aligned}D_0\text{ contains two translates $D_1,D_2$ of }\Phi_\A(D_0)
\\
\text{of distance at least }\vre_0\text{  from each other.} \quad \end{aligned}}
We claim Alice 
has a winning strategy for $\A> a''$. 

Given $T > 0$, define $Z_T \df \{g_tz : 0 \le t \le T\}$. This is a compact curve in $\ggm$,
with the crucial 
property that for any $x\in Z_T$, the tangent vector to $Z_T$ at $x$
(the flow direction) is 
orthogonal to the tangent space to $Hx$ at $x$. This is exploited in
the following

\begin{lem}\name{transv} For any $z\in\ggm$ and $T >0$ there exist
$\delta  = \delta(z,T)> 0$ 
with the following property: for any $ 0 < \vre \leq \delta$  
and any $x\in\ggm$, the intersection of $B_H(\delta)x$ with 
$Z_T^{(\vre/8)}$ has diameter at most $\vre$. 
\end{lem}

\begin{proof} \commargin{Barak: I wrote a new proof, see if you like
it. Dima: I like it, but please check my numerous minor corrections.}
Let $\goth{l}$ be the subspace of $\goth{g}$ perpendicular to $\goth{f}
\oplus \goth{h}$, so that we may write 
$$\goth{g} = \goth{f} \oplus \goth{l} \oplus \goth{h}.$$
The map 
$$
\psi: \goth{g} \to G, \ \ \psi(\vf + \vl+ \vh) = \exp(\vf)\exp(\vl) \exp(\vh)
$$
(where $\vf \in \goth{f}, \, \vl \in \goth{l}, \, \vh \in \goth{h}$) has
the identity map as its derivative at $0 \in \goth{g}$. Hence in a
sufficiently small neighborhood of $0$, it is a homeomorphism onto its
image and is 2-bilipschitz. Since the map $G \to \ggm, \ g \mapsto gx$
is a local isometry for each $x \in \ggm$, there is a
neighborhood $W=W_x$ of $0$ in $\goth{g}$ such that the map 
$$p_x: W \to \ggm, \ \vv \mapsto \psi(\vv)x$$ 
is injective, has an open image, and is 2-bi-Lipschitz. By compactness
of $Z_T$ we can choose $W$ uniformly for all $x \in Z_T^{(1)}$.
Choose $\delta<1$ small enough so that $B_G(2\delta) \subset
\psi(W)$. 

The order of factors in the definition of $\psi$ was chosen so that
if $y_1, y_2 \in \psi(W)$ are in the same $F$-orbit (i.e. $y_1
= g_ty_2$ for some $t$) then $\psi^{-1}(y_1) -
\psi^{-1}(y_2) \in \goth{f}$. Using this, we can now make $\delta$
and $W$ small enough so that if $B_H(\delta)x \cap Z_T^{(\delta)} \neq
\varnothing$, then also $Z_T \cap p_x(W) \neq
\varnothing$,  and $y_1, y_2 \in p_x(W) \cap Z_T$ implies
$p_x^{-1}(y_1) -p_x^{-1}(y_2) \in  \goth{f}$. 

If the claim is false then there are $\vre \leq \delta$, $x \in \ggm$,
$h_1, h_2 \in B_H(\delta)$, $y_1, y_2 \in p_x(W) \cap Z_T$ such that
$$\dist(h_ix, y_i) < \vre/8, \ \ i=1,2$$
and
$$\dist(h_1x, h_2x) > \vre.$$
Let $\vv_i = p_x^{-1}(y_i), \vh_i  = p_x^{-1}(h_ix) =
\exp^{-1}(h_i)$. Since the map $p_x$ is 2-bi-Lipschitz, 
\eq{eq: for triangle inequality}{
\|\vh_i- \vv_i\| < \vre/4 \ \ \mathrm{and} \ \|\vh_1- \vh_2\| > \vre/2\,.}
By the previous observation, $\vv_1, \vv_2$ belong to the same affine
subspace $\vv_0+ \goth{f}$. Without affecting \equ{eq: for triangle
inequality}, we can replace each $\vv_i$ with the point in
$\vv_0+\goth{f}$ closest to $\vh_i$. But since $\vh_i 
\in \goth{h}$ and $\goth{h}$ is orthogonal to $\goth{f}$, these two points
are the same, i.e.\ $\vv_1= \vv_2$. But now \equ{eq: for triangle
inequality} yields a contradiction to the triangle inequality. 
\end{proof}
Now given any choice of $\B>0$, an initial $t_1 > 0$ 
and any $B_1$ which is a right-translate of $\Phi_{t_1}(D_0)$, we are going to
find $\eta >0$ and $N\in\N$ such that for any choices made by Bob,
there are choices $A_i\subset B_i$ for Alice  
such that 
\eq{need to show pts}{
h \in A_{N+k},\, k \in\N \ \ \implies \ \dist( g_{kT}hx,Z_{T}) > \eta\,,
}
where $T = \A + \B$. It follows from \equ{need to show pts} that the
closure of the trajectory $F^+h_{\infty}x$, where $h_{\infty}$ is the 
intersection point \equ{int}, does not contain $z$. Indeed, if $g_{t_i}h_{\infty}x$
converges to  $z$, then we would have $\dist(g_{\ell_iT}h_{\infty}x,g_{t_i'}z) \to 0$
where $t_i = \ell_iT - t_i'$ with $0 \le t_i' < T$ and $\ell_i\in\N$, 
contradicting \equ{need to show pts}.

To this end, let $N$ be an integer such that the diameter of
$\Phi_{NT}(D_0)$ is less than $\delta =  \delta(z,T)$ as in Lemma \ref{transv}, let 
$$\lambda \df \inf_{h_1,h_2\in D_0, h_1 \ne h_2} \frac
{\dist\big(\Phi_{NT}(h_1),\Phi_{NT}(h_2)\big)}{\dist(h_1,h_2)}\,
$$
be the Lipschitz constant for $\Phi_{NT}|_{D_0}$, and choose 
$$\vre  = \min \big(\tfrac12 \lambda\vre_0, \delta \big), \ \ \ \eta = \vre/8,
$$
where $\vre_0$ is as in \equ{two choices}.

We now describe Alice's strategy. For her first $N$ moves she makes
arbitrary choices. For $k \in \N$, suppose that Bob's  
choice is $ B_{N+k} = \Phi_{(N+k)T}(D_0)h$. Let $D_1,D_2$ be as in
\equ{two choices}. Then $E_i = \Phi_{(N+k)T}(D_i)h$, $i = 1,2$,
are both valid options for Alice's 
next move.
Also, 
\eq{transl ei}{g_{kT}E_ix = g_{kT}\Phi_{(N+k)T}(D_i)hx =
\Phi_{NT}(D_i)g_{kT} hx\,,}
and $\Phi_{NT}(D_i)$ are two elements in $\cd_{NT + \A}$ inside $\Phi_{NT}(D_0)$
of distance at least $\lambda\vre_0$ from each other. Furthermore, the diameter
of $\Phi_{NT}(D_0)$ is less than $\delta$, so by Lemma \ref{transv}, 
the intersection of $$\Phi_{NT}(D_0)g_{kT}hx = g_{kT}B_{N+k}x$$
with $Z_T^{(\vre/8)}$ has diameter less than $\frac12 \lambda\vre_0$.
Therefore at least one of the sets \equ{transl ei} is disjoint from
$Z_T^{(\vre/8)}$, i.e.\ is at distance at least $\eta$ 
from $Z_T$. Alice chooses the corresponding $E_i$, and \equ{need to show pts} holds. 
\end{proof}

We conclude the section with the

\begin{proof}[Proof of Corollary \ref{cor: D}] In view of  the
countable intersection property of winning sets, see Theorem \ref{thm: countable general}(a),
the choice $\A> a_0 \df
\max(a',a'')$, where $a', a''$ are the constants furnished by
Theorems \ref{thm: C} and \ref{thm: B} respectively, shows that the set \equ{efzinftyh}
is an $\A$-winning set for the MSG induced by $\FF$.
\end{proof}

\ignore{
Using the countable intersection properties of winning sets, it is now immediate that
whenever
$X_0,Z$ are countable subsets of $\ggm$ and $H$ is as in Thereom \ref{thm: B more general}
 the set
$$\bigcap_{x\in X_0}\left\{h \in H :  gx\in E(F^+,Z)\right\}$$
is also an $\A$-winning set as long as $\A > a''$. Furthermore, if we take $G,\Gamma$ 
as in Theorem \ref{thm: C} and $H$ to be  $(F^+,\Gamma)$-expanding 
and let $a_0 = \max(a',a'')$, then for any  $\A > a_0$ the same can be said about the set
\equ{efzinftyh} with $H^+$ replaced by $H$. Since $H = H^+$ gives an
example of an
 $(F^+,\Gamma)$-expanding subgroup, this furnishes a proof of
Corollary \ref{cor: D}.}

\section{Proof of Theorem \ref{cor: F}}\name{margconj}
The goal of this section is to reduce Theorem \ref{cor: F} to Corollary \ref{cor: D}.
We will do it in several steps. We denote by
\eq{defboxdim}{\underline{\dim}_B (X)\df \liminf_{\vre \to 0} \frac{\log
N_X(\vre)}{-\log \vre},\ \ 
\overline{\dim}_B(X) \df \limsup_{\vre \to 0} \frac{\log N_X(\vre)}{-\log \vre}
}
the lower and upper box dimension of a metric space $X$, 
where $N_X(\vre)$ is the smallest number of sets of diameter $\vre$
needed to cover $X$. Note that one always has
 \eq{boxdim bounds}{
\dim(X)\le \underline{\dim}_B(X) \le  \overline{\dim}_B(X) \,.}

We will need two classical facts about \hd:

\begin{lem}\name{marstrand} {\rm (Marstrand Slicing Theorem, see \cite[Theorem 5.8]{Falconer} or
\cite[Lemma 1.4]{KM})}
Let $X$ be a Borel subset of a manifold 
such that $\dim(X)\ge \alpha$, let $Y$ be a metric space, and let $Z$ be a subset 
of the direct product $X\times Y$ such that
$$
\dim\big(Z\cap (\{x\}\times Y)\big)\ge\beta
$$
for all $x\in X$.
Then 
$\dim(Z) \ge\alpha+\beta$.
\end{lem}

\commargin{need to check a reference.. Dima: done it.}
\begin{lem}\name{wegmann} {\rm (Wegmann Product Theorem, see \cite{weg} or \cite[Formulae 7.2 and 7.3]{Falconer})} 
 For any two metric spaces $X,Y$, 
$$
\dim(X) + \dim (Y) \le \dim \left(X\times Y \right)\le \dim(X) +
\overline{\dim}_B(Y)\,.
$$
\end{lem}

\subsection{Two-sided orbits} \name{2sided}
First let us show that Corollary \ref{cor: D}
implies that there is a thick set of points with $F$-orbits
bounded and staying away from a countable set.

\begin{prop}\name{prop: twosided}  Let $G$, $\Gamma$ and $F$ 
be as in Theorem \ref{thm: C}, and let $Z$ be
a countable subset of $\ggm$. Then  $E(F,Z)  \cap E(F,\infty)$ is thick.
\end{prop}
\commargin{Barak: rewrote this proof. Dima: OK, but again please check my corrections here and elsewhere.}
\begin{proof}
The argument follows the lines of  \cite[\S 1]{KM}. Recall that we denoted by $H^+$ the \ehs\ with respect to $g_1$.
Similarly one can consider the group $H^-$, which is
expanding horospherical  with respect to $g_{-1}$,
and another group $H^0$ such that the multiplication map
\eq{mult}
{(h^-,h^+,h^0)\mapsto
h^-h^+h^0,
\ \ H^-\times H^+ \times H^0\to G \,,}
is locally (in the neighborhood of identity) very close to an isometry.
More precisely, given a nonempty open $W\subset\ggm$, choose a point $x\in W$ and take
$U\subset G$ of the form $U^- U^+ U^0$, where $U^-$, $U^+$  and $U^0$ are
neighborhoods of identity in $H^-$, $H^+$ and $H^0$ respectively, such that the
map 
\equ{mult} 
is bi-Lipschitz on $U$, the orbit map $\pi_x:U\to\ggm$, $g\mapsto gx$, is injective
and its image is contained in $W$. Then it is enough to show that
\eq{enough}{
\dim\left(\big\{g\in U : gx\in E(F,Z)  \cap E(F,\infty) \big\}\right) = \dim(G)\,.
}
Denote
$$C^{\pm} \df \{g\in U : gx\in E(F^{\pm},Z)  \cap E(F^{\pm},\infty)\}\,.$$
First we fix $h^0 \in U^0$. From Corollary \ref{cor: D} and Theorem
\ref{thm: countable general}(b) it follows that 
\eq{follows1}{\dim(C^+\cap U^+h^0) = \dim (H^+) \,. }
We now claim that for any $g\in C^+$, the set $U^-g \smallsetminus C^+$ is
at most countable. Indeed, because orbits of $H^-$ are stable manifolds
with respect to the action of $F^+$, for any $h^-\in U^-$ the 
trajectories $F^+gx$ and $F^+h^-gx$ have the same sets of accumulation points.
Since  $g\in C^+$, we conclude that either
 $h^-gx\in E(F^+,Z)  \cap E(F^+,\infty)$, or $g_th^-gx \in Z$ for some
$t\ge 0$, which can only happen for countably many $h^-$.

Applying  Corollary \ref{cor: D} and Theorem \ref{thm: countable general}(b)
 (with $F^-$ in place of $F^+$ and $H^-$ in place
of $H^+$), one concludes that for all $h\in C^+\cap U^+ h^0,$
$$\dim( \left\{h^- \in U^-: h^-h \in C^-
\right \} ) = \dim (H^-)\,.$$
Using the claim, and since a countable set of points does not affect
Hausdorff dimension, we find  that for all $h\in C^+\cap U^+ h^0,$
\eq{follows2}{\begin{split}
& \dim\left( \left\{h^- \in U^- : h^-h \in C^+ \cap C^-\right\}\right) \\
= & \dim\left( \left \{h^- \in U^-: h^-hx \in E(F,Z) \cap E(F,\infty)
\right \} \right)= \dim ( H^-) \,. 
\end{split}
}
Using Lemma \ref{marstrand} and \equ{follows1},\equ{follows2}
we find that 
$$\dim \, \left( (U^- \times U^+)h^0 \cap C^+ \cap C^- \right) =
\dim (H^+) + \dim( H^-)\,.$$
Now using Lemma \ref{marstrand} again (with $X = U^0$) we obtain
\equ{enough}. 
\end{proof}

\subsection{Entropy argument} \name{entropy}
In this subsection we will show that the requirement $\dim (\cl{Fx} ) <
\dim (G)$ is fulfilled automatically whenever $Fx$ is a bounded and nondense orbit
on $\ggm$, where $G$, $\Gamma$ and $F$ are as in Theorem \ref{thm: C}.
Combined with the result of the previous section, this will establish

\begin{prop}\name{prop: smalldim irr}  Let $G$, $\Gamma$ and $F$ 
be as in Theorem \ref{thm: C}, and let $Z$ be
a countable subset of $\ggm$. Then the set \equ{efzinfty} is thick.
\end{prop}

The argument contained in this subsection was explained to us by
Manfred Einsiedler and Elon Lindenstrauss. It relates the Hausdorff
dimension to the topological and metrical entropy of the action of
$F^+$. For a detailed recent exposition we refer the reader to the
survey \cite{EL}.

\begin{prop}
\name{entr}
Suppose $G$ is a semisimple Lie group with no compact factors,
$\Gamma$ is an irreducible lattice, and $x \in \ggm$ such that
$F^+x$ is bounded but not dense. Then necessarily $\dim( \cl{F^+x}) < \dim (G)$. 
\end{prop}

We will need the following
statement, whose proof is postponed to the end of this section:
\begin{lem}\name{lem: relating dim and entropy}
Suppose $g_1$ is a non-\qu \ element of $G$ 
and $K \subset \ggm$ is compact and $g_1$-invariant, with $\dim (K) = \dim
(G)$. Then the topological entropy $h_{\mathrm{top}}(g_1)$ \commargin{Dima: added ``on $K$"}
 of the $g_1$-action on $K$
is equal to 
the measure theoretic entropy of the action of $g_1$ on $\ggm$. 

\end{lem}

\begin{proof}[Proof of Proposition \ref{entr} assuming Lemma \ref{lem:
relating dim and entropy}] Suppose by contradiction that 
$\dim (K) = \dim (G)$, where $K = \cl{F^+x}$ is compact. By Lemma \ref{lem: relating dim and
entropy} we find that the topological entropy $h_{\mathrm{top}} (g_1) $
is equal to $h_{\max}$, the measure theoretic entropy of the action of
$g_1$ on $\ggm$. By the variational
principle \cite[Prop. 3.21]{EL} there are $g_1$-invariant measures $\mu_i$
on $K$ with entropy $h(\mu_i)$ tending to $h_{\max}$. Since $K$ is
compact, we can take a weak-$*$ limit $\nu$ of these measures, then by
\cite[Prop. 3.15]{EL} we have $h(\nu) = h_{\max}$, so by \cite[Thm. 7.9]{EL}, 
$\nu$ is invariant under $H^+$. Since the stabilizer of a measure is a
group, $\nu$ is also invariant under $g_1^{-1}$, and since the entropy
of a map is the same as that of its inverse, we see that $\nu$ is also
invariant under the horospherical subgroup $H^-$ of $g_1^{-1}$. Let
$G_0$ be the subgroup of 
$G$ generated by $H^+$ and $H^-$; it is sometimes called the {\sl
Auslander subgroup\/} corresponding to $g_1$, see
\cite[\S2.1.a]{handbook}. It is normal in $G$, so by Proposition 
\ref{prop: since Gamma irreducible} it acts uniquely ergodically on
$\ggm$. Therefore $\nu$ is the Haar measure which is of full support. But $\nu$ is
supported on $K$, so $K = \ggm$, contrary to assumption. 
\end{proof}

\begin{proof}[Proof of  Lemma \ref{lem:
relating dim and entropy}]
We will use the following formula for $h_{\max}$ (see e.g. \cite{EL}):
\eq{eq: formula for entropy}{
h_{\max} = \log \big|\det \Ad(g_1)|_{\goth{h}^+} \big| =
\sum_{\alpha \in \Phi^+} \dim (\goth{g}_{\alpha} )\,  \alpha(a_1),}
where $g_1=k_1u_1a_1$ is the 
decomposition of $g_1$ into its compact,
semisimple and unipotent parts as in Proposition \ref{prop: Jordan},
$\goth{h}^+$ is the Lie algebra of the expanding horospherical
subgroup $H^+$, $\Phi$ is the set of roots, and 
$\Phi^+ = \{\alpha \in \Phi : \alpha(a_1)>0\}. $
Recall that  $x_1, x_2 $ are said to be {\sl $(n, \vre)$-separated\/}
if there is $j \in \{0, \ldots, n-1\}$ such that $d(g_1^jx_1, g_1^jx_2)
\geq \vre$. Then one has 
$$h_{\mathrm{top}}(g_1) = \sup_{\vre>0} \lim_{n \to \infty}\frac{\log
N_K(n, \vre)}{n}\,,$$ 
where $N_K(n,\vre)$ is the maximal cardinality of a set of
mutually $(n, \vre)$-separated points of $K$. 
Suppose by contradiction that $h_{\mathrm{top}}(g_1) = h_{\max} -
\delta_1$ for some $\delta_1>0$.  Then there is $\delta_2>0$, and for
any $\vre$, there exists a sequence $n_k \to \infty$ so that for all $k$, there
are at most 
$$M_k\df e^{n_k( h_{\max} - \delta_2)}$$ 
mutually $(n_k, \vre)$-separated points of $K$. This means that $K$ is
covered by $M_k$ sets of the form $B(x, n_k, \vre)$, where 
$$
B(x, n, \vre) \df \bigcap_{j=0}^{n-1} g_1^{-j}\left(B(g_1^jx, \vre)
\right). 
$$
Applying an argument similar to that of \cite[Prop. 8.3]{EKL}, using
the fact that $K$ is compact and the exponential map is  
bi-Lipschitz in a sufficiently small neighborhood of the identity, we
find: for any small enough $\vre$ there is $\vre'>0$ so that if we
represent $y = \exp(\vv)x$, 
with $\vv \in \goth{g}, \, \|\vv\| \leq \vre'$ and $x,y\in K$, \commargin{Dima: it used to be $g_1^ny, g_1^nx \in
K$ for all $n$, and $K$ was not defined at all; I am not sure what you meant, please check!}
then for the first index $j$ for which 
$\|\Ad^j(g_1)\vv\| \geq \vre'$, we will have $d(g_1^jx, g_1^jy)
\geq \vre$. This implies that 
any set $B(x, n_k, \vre)$ as above is covered by a set 
$\exp\big(C(n_k, \vre')\big)x$, where 
$$C(n, \eta) \df \bigcap_{j=0}^{n-1} \{\vv \in \goth{g}: \|\Ad^j(g_1)\vv \|
\leq \eta\}. 
$$
We may simplify calculations by assuming that the norm in the above
formula is the sup norm with respect to a basis 
$$\{\vv^{(\alpha)}_i :
\alpha \in \Phi, i=1, \ldots, \dim( \goth{g}_\alpha)\}$$ 
of simultaneous eigenvectors for the action of $\Ad(a_1)$ on
$\goth{g}$. Since $\Ad^j(g_1) = \Ad^j(k_1n_1) \circ \Ad^j(a_1)$ and 
$\Ad^j(k_1n_1)$ is a quasi-unipotent map, whose norm grows at most polynomially in $j$, for any
sufficiently small $\delta_3>0$ and for any sufficiently large $j$, 
$\|\Ad^j(g_1) \vv \|$ is bounded above by 
$$ 
\max
\left( \left\{e^{j(\alpha(a_1)-\delta_3)}|c^{(\alpha)}_i| : \alpha \in \Phi^+
\right\} \cup \left \{e^{j\delta_3} 
|c^{(\alpha)}_i| : \alpha \in \Phi \sm \Phi^+ 
\right\}
\right),
$$
where 
$$
\vv = \sum_{\alpha \in \Phi} \sum_{i=1}^{\dim (\goth{g}_\alpha)}
c^{(\alpha)}_i \vv^{(\alpha)}_i.
$$
Therefore $C(n, \eta) \subset D(n,\eta)$, where the latter is defined
as
$$
\left\{\vv : \forall \alpha \in \Phi^+, \, 
|c^{(\alpha)}_i| \leq e^{n(\delta_3 -\alpha(a_1))} \eta
\mathrm{\ and \
} 
\forall \alpha \in \Phi \sm \Phi^+, \, 
|c^{(\alpha)}_i| \leq e^{n\delta_3} \eta
 \right\}.
$$
We cover $D(n, \eta)$ by cubes in $\goth{g}$ whose sidelength is the
smallest of these dimensions, namely $e^{n(\delta_3 -
\alpha_{\max})}\eta$, where 
$$\alpha_{\max} = \max\{\alpha(a_1): \alpha \in \Phi^+\}.$$
Then the number of these cubes is at most
$$
M'_n = \exp\left[ n\left(\alpha_{\max} \, \dim( G)  + 2\delta_3
\dim (\goth{g}^{\leq}) - \sum_{\alpha \in \Phi^{+}} \dim (\goth{g}_{\alpha} )
\, \alpha(a_1) \right) \right], 
$$
where $\goth{g}^{\leq} = \bigoplus_{\alpha(a) \leq 0} \goth{g}_{\alpha}.$
Using the images of these cubes under the exponential map, we find
that for some constant $c$, we have a
covering of $K$ by sets whose diameter is at most $ce^{n_k(\delta_3 -
\alpha_{\max})},$ and whose number is bounded by $ M_{n_k} M'_{n_k} $,
i.e.\ is not greater  than 
\[
\begin{split} 
 & \exp \left[
n_k\left( h_{\max} - \delta_2+ \alpha_{\max} \, \dim (G) +
2\delta_3 \dim (\goth{g}^{\leq}) - \sum_{\alpha \in
\Phi^{+}} \dim (\goth{g}_{\alpha}) \alpha(a_1) \right)
\right] \\ 
& \under{\equ{eq: formula for entropy}}{=}
c \exp \left[n_k \left(
-\delta_2+ \alpha_{\max} \,  \dim( G) +
2\delta_3 \dim( \goth{g}^{\leq})
\right)
\right].
\end{split}
\]
Thus we have proved that $$
N_K(ce^{n_k(\delta_3 -
\alpha_{\max})}) \leq c \exp \left(n_k \left(\alpha_{\max} \dim (G)
-\delta_2+ 2 \delta_3 \dim (\goth{g}^{\leq})
\right)
\right).
$$
This, in view of \equ{defboxdim}, implies
\[
\begin{split}
\dim (K) & \leq \frac{\log c + n_k\left(\alpha_{\max} \, \dim (G) - \delta_2+ 2
\delta_3 \dim (\goth{g}^{\leq})\right)}{-\log c+ n_k(\alpha_{\max} - \delta_3)
} \\
&  \to_{k \to \infty} \frac{\alpha_{\max} \dim (G) + 2\delta_3 \dim (\goth{g}^{\leq}) - 
\delta_2}{\alpha_{\max} - \delta_3 } \\
& \to_{\delta_3 \to 0} \dim( G)
-\delta_2/\alpha_{\max} < \dim (G)\,,
\end{split}
\]
contradicting the assumption of the lemma. \end{proof}

\subsection{Completion of the proof} \name{compl} Here we reduce Theorem \ref{cor: F}
to the set-up of  Theorem \ref{thm: C} and Proposition \ref{prop:
smalldim irr}. There are several additional steps left, as we will see
below. 

\begin{proof}[Proof of  Theorem \ref{cor: F}] Recall that we have taken $G$ to be an arbitrary Lie group
and $\Gamma$ a lattice in $G$. It is easy to see that $G$ can be assumed to be connected, since $F$
happens to be a subgroup of the connected component $G^0$ of identity, and connected components 
of $G/\Gamma$ are copies of $G^0/(G^0\cap \Gamma)$.
 
Now let  $R(G)$ be the radical of $G$.
Then $G/R(G) = G_0\times
\widehat G$, where $G_0$ is compact and $\widehat G$ is connected semisimple without
compact factors. 
Let $\pi:G\to\widehat G$ be the canonical projection, then (see \cite[Chapter 9]{Raghunathan} 
or \cite[Lemma 5.1]{Dani-inv})
 $\hat\Gamma\df 
\pi(\Gamma)$ is a lattice in $\hat G$. Also, recall that $F$ is assumed to be
absolutely non-quasiunipotent; clearly the same can be said about $\pi(F)$. 
If we denote by $\bar\pi$ the induced map of
\hs s $\ggm\to\widehat G/\widehat \Gamma$, then it is known (see the
second part of \cite[Lemma 5.1]{Dani-inv}) that $\bar \pi$ has compact fibers and
$ \Ker \pi\,\cap\,\Gamma\subset \Ker \pi$ is a
uniform lattice. 
Suppose we knew that the set
\eq{efzinftynew}{\{ x\in E\big(\pi(F),\bar \pi(Z)\big)  
\cap E\big(\pi(F),\infty\big):\dim(\overline{\pi(F)x})
 < \dim(\widehat G)\}
}
is thick.  We claim that the $\bar \pi$-preimage of the above set
is contained in the set \equ{efzinfty};  
 thus the latter is also thick. Indeed, if $\bar \pi(x)$ belongs to $ E\big(\pi(F),\infty\big)$,
then  $Fx$ is bounded by the compactness of the fibers of $\bar \pi$. Also, by the continuity
of $\bar \pi$,
if $z$ is in the closure of $Fx$, then $\bar \pi(z)$ must be in the closure of $\pi(F)\bar \pi(x) = 
\bar \pi(Fx)$. Finally, the orbit closure $\overline{F x}$ is contained in the preimage
of $\overline{\pi(F)\bar \pi(x)}$ and therefore has codimension not less than the codimension
of  $\overline{\pi(F)\bar \pi(x)}$, in view of Lemma
\ref{wegmann}. 

\smallskip

Thus we have reduced  Theorem \ref{thm: C} to the case when $G$  is connected semisimple without
compact factors. Next note that without loss of generality we can assume that 
the center $C(G)$ of $G$ is trivial. Indeed, let us denote the
quotient group $G/C(G)$ by $G'$, the homomorphism
$G\to G'$ by $p$,  and the induced map
$\ggm\to G'/p(\Gamma)$ by $\bar p$.
Since $\Gamma  C(G)$ is discrete \cite[Corollary 5.17]{Raghunathan},
$p(\Gamma)$ is also discrete, hence the quotient $C(G)/\big(\Gamma\cap C(G)\big)$ is finite.
This means that 
$(\ggm,\bar p)$ is a finite covering of $ G'/p(\Gamma)$, 
and an argument similar to that of  the previous step completes the reduction. 

\smallskip

Finally we are ready to reduce to the set-up of  Proposition \ref{prop: smalldim irr}.
 Let $G_1,\dots,G_\ell$ be connected normal subgroups of $G$ such
that $G=\prod_{i=1}^{l}G_i$, $G_i\cap G_j = \{e\}$ if $i\ne j$, 
$\Gamma_i = G_i\cap\Gamma$ is an irreducible lattice in $G_i$ for each
$i$,  and $\prod_{i=1}^\ell\Gamma_i$ has finite index in $\Gamma$. 
Denote by $p_i$ the projection $G\to G_i$; we know that for
any $i\in\{1,\dots,\ell\}$, the group $p_i(F)$ is either trivial or not quasiunipotent.
Also define maps
$$
\ggm \stackrel{\tilde p}{\to} \prod_{i=1}^\ell (G_i/\Gamma_i) = G/\prod_{i=1}^\ell p_i(\Gamma)
 \stackrel{\bar p_i}{\to} G_i/\Gamma_i\,,
$$
and let $\tilde p_i \df \tilde p\circ \bar p_i$.
 Applying  Proposition \ref{prop: smalldim irr}
to each of the spaces $G_i/\Gamma_i$, we conclude that for all $i$, the sets
$$A_i \df \{ x\in E\big(p_i(F),\tilde p_i(Z)\big)  
\cap E\big(p_i(F),\infty\big):\dim(\overline{p_i(F) x})
 < \dim(G_i)\}
$$
are thick. We now claim that the set \equ{efzinfty} contains $\tilde p^{-1}(A_1\times\dots\times A_\ell  )$.
 Indeed, if $Fx$ is unbounded, then so is $F\tilde p(x)$, and hence at least one
of its projections onto  $G_i/\Gamma_i$. Likewise,
if $z$ is in the closure of $Fx$, then clearly  its projection is in the closure
of the projection of $Fx$.  Finally, the orbit closure $\overline{F x}$ is contained in 
$\tilde p^{-1}\big(\prod_{i=1}^\ell   \overline{p(F)\tilde  p_i(x)}\big)$;
and as long as the codimension of at least one of the sets $\overline{p(F)\tilde  p_i(x)}$ is positive,
the same can be said about $\overline{F x}$, again by Lemma \ref{wegmann}. We conclude that
the set \equ{efzinfty} is thick, as claimed. \end{proof}

\section{Concluding remarks}\name{concl}
Let $m,n$ be positive integers,  and
let  $\vr\in\R^m$ and $\vs\in\R^n$ be as in
\equ{def rs}.
One says that  \amr\ (interpreted as a system of $m$ linear forms in $n$ variables) is  
 {\sl $(\vr,\vs)$-\ba\/}, denoted by $Y\in\BA(\vr,\vs)$, if
 \eq{def bars}{ \inf_{\vp\in\Z^m,\,\vq\in\Z^n\nz} \max_i |Y_i\vq - p_i|^{1/r_i} \cdot
\max_j|q_j|^{1/s_j}    > 0\,,
}
where $Y_i$, $i = 1,\dots,m$ are rows of $Y$ (linear forms $\vq\mapsto Y_i\vq$). The choice 
of constant weights $r_i \equiv 1/m$ and $s_j \equiv 1/n$ corresponds to the
classical  notion of {\sl \ba\/} systems.
It has been observed by Dani \cite{Dani-div}  
 that \amr\  is \ba\ if and only if $u_Y\Gamma \in E(F^+,\infty)$ where
$G$, $\Gamma$ and $F$ are as in  \equ{case1}, $k = m+n$, and $u_Y$ is as in \equ{def uy}.
This way he could use Schmidt's result on thickness of the set of \ba\ matrices
to construct bounded trajectories in this particular case.

Now, more generally,  consider $G,\Gamma$ as in   \equ{case1}, 
and let $F = F_{\vr,\vs}$ be as in 
\equ{new defn gt}.
 A rather straightforward generalization of Dani's result \cite{K-matrices} shows
 that
$Y\in\BA(\vr,\vs)$  if and only if   $u_Y\Gamma \in E(F_{\vr,\vs}
 ^+,\infty)$.
Thus Theorem \ref{thm: more general}, in view of Proposition \ref{prop: epimorphic},
implies

\begin{cor}\name{cor: E} $ \BA(\vr,\vs)$
is winning for the MSG induced by the semigroup of contractions
$\Phi_t : (y_{ij}) \mapsto ( e^{- (r_i + s_j)t }y_{ij})$ of $\mr$. 
 \end{cor}

The case $n=1$ of the above corollary is the main result of \cite{games}, which was proved via 
a variation of Schmidt's methods, not using homogeneous dynamics. The fact that the
set $ \BA(\vr,\vs)$ is thick was known before, see \cite{PV-bad} for
the case $n=1$ and \cite[Corollary 4.5]{di} for the general case. 
Note that a famous problem dating back to 
W.\ Schmidt \cite{Schmidt:open} is to determine whether or not for {\it different\/} pairs
$(\vr,\vs)$ and $( \vr',\vs')$, the intersection of $\BA(\vr,\vs)$ and $\BA(\vr',\vs')$ could be empty. Schmidt
conjectured that the intersection is  non-empty in the special case $n = 2$, $m = 1$, $\vr = (\frac13, \frac23)$
and $\vr' = (\frac23, \frac13)$. Recently this conjecture was proved in a much stronger form by
Badziahin, Pollington and Velani \cite{BPV}: they established that the intersection of countably many 
of those sets has \hd\ $2$, as long as the weights $(r_1,r_2)$ are bounded away from the endpoints
of the interval $\{r_1 + r_2 = 1\}$. This was done without using Schmidt games.
Unfortunately, the results of the 
present paper do not give rise to any further progress related to Schmidt's Conjecture,
since each pair $(\vr,\vs)$ defines a different MSG, and we are unable to show that 
winning sets of different games must have nonempty intersection. 


However, our technique can be applied to a similar problem of constructing points
whose orbits under different one-parameter semigroups $F_j^+$ stay away from a countable subset
of $\ggm$. Namely, the following can be established:

\begin{cor}\name{cor: S} Let $G$ be a Lie group, $\Gamma$ a
discrete subgroup of $G$, $Z$ a countable subset of $\ggm$, and  let 
$\{F_j\}$ be a countable collection
of one-parameter subgroups of $G$ such that there exists
 a one-dimensional
subgroup $H$ of $G$ which is contained in the expanding horospherical
subgroups corresponding to $F_j^+$ and is normalized by $F_j$ for each
$j$. 
Then  the set
$\bigcap_{j}\ E(F_j^+,Z)$ 
is thick.
\end{cor}
 
 \begin{proof} By Theorem \ref{thm: B more general}, for any $x\in\ggm$, each of the sets
\eq{efzj}{\{h\in H : hx\in E(F_j^+,Z)\}} is a winning set for the modified Schmidt game induced by the 
 semigroup of contracting automorphisms
 of $H$  given by the conjugation by elements of $F_j^+$. However, since $H$ is one-dimensional,
 one can reparameterize each of the groups $F_j$ so that they all give rise to the same semigroup 
 $\FF = \{\Phi_t|_H\}$ and induce the same modified game (isomorphic to the original Schmidt's game). 
 Then one can choose an interval $D_0$ and $a'' > 0$ such that  \equ{two choices}
 is satisfied, and,  as in the proof of Theorem \ref{thm: B more general},  conclude that each of the sets \equ{efzj} is 
$a$-winning for any $a > a''$; thus the claim follows from Theorem \ref{thm: countable general} and Lemma \ref{marstrand}. \end{proof}

An example of a situation described in the above corollary  is furnished
by the action of  groups $F = F_{\vr,\vs}$ as in 
\equ{new defn gt} on homogeneous spaces $\ggm$ of $G = \SL_{m+n}(\R)$, where 
$\Gamma$ is any discrete subgroup of $G$ and $(\vr,\vs)$ is as in
\equ{def rs}. It is easy to see that the group $$\{u_y \df I_{m+n} + yE_{1,m+n} : y\in \R\}\,,$$
where $E_{1,m+n}$  stands for 
the matrix with $1$ in the upper-right corner and $0$ elsewhere,
satisfies the assumptions of the corollary. Thus for any countable set of pairs
$(\vr_j,\vs_j)$, any countable subset  $Z$ of $\ggm$ and any $x\in\ggm$, 
the set $\bigcap_{j}\{y\in\R:  u_yx\in E(F_{\vr_j,\vs_j}^+,Z)$ is winning.


\end{document}